\documentclass{amsart}

\usepackage{amssymb}
\usepackage{amsmath}
\usepackage{amsfonts}
\usepackage{geometry}
\usepackage{bbm}
\usepackage{hyperref}
\usepackage{tikz}
\usepackage{mathrsfs}
\usepackage{shuffle}
\usepackage[utf8]{inputenc}
\usepackage{mathtools}
\usepackage{parskip} 
\usepackage{todonotes}
\usepackage{multirow}

\usetikzlibrary{matrix,arrows,decorations.pathmorphing}
\tikzset{thick,level distance=3mm,sibling distance=6mm}

\newcounter{cprop}[section]

\newtheorem{theorem}[cprop]{Theorem}
\newtheorem*{theorem*}{Theorem}

\theoremstyle{plain}

\newtheorem{corollary}[cprop]{Corollary}
\newtheorem*{corollary*}{Corollary}

\newtheorem{lemma}[cprop]{Lemma}
\newtheorem{proposition}[cprop]{Proposition}

\numberwithin{equation}{section}

\theoremstyle{definition}
\newtheorem{definition}[cprop]{Definition}
\newtheorem{example}[cprop]{Example}

\theoremstyle{remark}
\newtheorem{remark}[cprop]{Remark}

\newcommand{\E}{\mathbb{E}}
\renewcommand{\P}{\mathbb{P}}

\newcommand{\R}{\mathbb{R}}

\newcommand{\N}{\mathbb{N}}
\newcommand{\Z}{\mathbb{Z}}
\newcommand{\X}{\mathbb{X}}
\renewcommand{\L}{\mathbb{L}}
\newcommand{\Y}{\mathbb{Y}}
\newcommand{\blue}[1]{\textcolor{blue}{#1}}

\newcommand{\tlog}{\log^{\otimes}}
\newcommand{\texp}{\exp^{\otimes}}
\newcommand{\sexp}{\exp^{\shuffle}}

\newcommand{\vertiii}[1]{{\left\vert\kern-0.25ex\left\vert\kern-0.25ex\left\vert #1
    \right\vert\kern-0.25ex\right\vert\kern-0.25ex\right\vert}}

\renewcommand{\hat}{\widehat}
\renewcommand{\tilde}{\widetilde}

\begin{document}
\title[Optimal stopping signatures]{Optimal stopping with signatures}

\author{C. Bayer}
\address{Christian Bayer \\ Weierstrass Institute \\ Berlin, Germany}
\email{christian.bayer@wias-berlin.de}

\author{P. Hager}
\address{Paul Hager \\
Institut f\"ur Mathematik, Technische Universit\"at Berlin, Germany}
\email{phager@math.tu-berlin.de}

\author{S. Riedel}
\address{Sebastian Riedel \\
Institut f\"ur Analysis, Leibniz Universit\"at Hannover, Germany}
\email{riedel@math.uni-hannover.de}

\author{J. Schoenmakers}
\address{John Schoenmakers \\ Weierstrass Institute \\ Berlin, Germany}
\email{john.schoenmakers@wias-berlin.de}



\begin{abstract}
  We propose a new method for solving optimal stopping problems (such as American option pricing in finance) under minimal assumptions on the underlying stochastic process $X$.
  We consider classic and randomized stopping times represented by linear and non-linear functionals of the \emph{rough path signature} $\X^{<\infty}$ associated to $X$, and prove that maximizing over these classes of \emph{signature stopping times}, in fact, solves the original optimal stopping problem. Using the algebraic properties of the signature, we can then recast the problem as a (deterministic) optimization problem depending only on the (truncated) expected signature $\E\left[ \X^{\le N}_{0,T} \right]$. By applying a deep neural network approach to approximate the non-linear signature functionals, we can efficiently solve the optimal stopping problem numerically.
  The only assumption on the process $X$ is that it is a continuous (geometric) random rough path. Hence, the theory encompasses processes such as fractional Brownian motion, which fail to be either semi-martingales or Markov processes, and can be used, in particular, for American-type option pricing in fractional models, e.g. on financial or electricity markets. 
\end{abstract}

\maketitle

\setcounter{tocdepth}{1}
\tableofcontents

\section{Introduction}
The theory of \emph{rough paths}, see, for instance, \cite{LCL07,FV10,FH14}, provides a powerful and elegant pathwise theory of stochastic differential equations driven by general classes of stochastic processes -- or, more precisely, rough paths. One of the benefits of the theory is that the resulting solution maps are continuous rather than merely measurable as in the It\=o version. This property has lead to many important theoretical progresses, most notably perhaps Hairer's theory for singular non-linear SPDEs.

In addition to the theoretical advances, tools from rough path analysis -- specifically, the path \emph{signature} -- play an increasingly prominent role in applications, most notably in \emph{machine learning}, see, e.g., \cite{AGGLS18}. Intuitively, the signature $\X^{<\infty}$ of a path $X:[0,T] \to \R^d$ denotes the (infinite) collection of all iterated integrals of all components of the path against each other, i.e., of the form
\begin{equation*}
  \int_{0 < t_1 < \cdots < t_n < T} dX^{i_1}_{t_1} \cdots dX^{i_n}_{t_n},
\end{equation*}
$i_1, \ldots, i_n \in \{1, \ldots, d\}$, $n \ge 0$.
To better understand the importance of the signature, let us first recall that the signature $\X^{<\infty}$ determines the underlying path $X$ (up to ``tree-like excursions''), which was first proved in \cite{HL10} for paths $X$ of bounded variation and later extended to less regular paths. This implies that, in principle, we can always work with the signature rather than the path. (A somewhat dubious proposition, as we merely replace one infinite dimensional object by another one.)
However, the signature is not an arbitrary encoding of the path. Rather, Lyons' universal limit theorem suggests that the solution of a differential equation driven by a rough path $X$ can be approximated with high accuracy by relatively few terms of the signature $\X^{<\infty}$. In that sense, an appropriately truncated signature can be seen as a highly efficient \emph{compression} of $X$, at least in the context of dynamical systems. And, indeed, there is now ample evidence of the power of the signature as a \emph{feature} in the sense of machine learning.

This paper is motivated by another recent application of signatures, namely the solution of stochastic optimal control problems in finance. We follow the presentation of \cite{KLP20}, where a signature-based approach for solving optimal execution problems is developed. In a nutshell, the strategy can be summarized as follows:
\begin{enumerate}
\item Trading strategies for execution of a position can be understood as (continuous) functionals $\phi(X|_{[0,t]})$ of the price path, and, hence, as functionals $\theta(\X^{<\infty}_{0,t})$ of the signature (at least approximately).
\item Taking advantage of the algebraic structure of the signature (see Section~\ref{sec:notat-basic-defin} below), we may efficiently approximate continuous functionals $\theta(\X^{<\infty}_{0,t})$ by \emph{linear} functionals $\langle l,\, \X^{<\infty}_{0,t} \rangle$, which further extends to the whole value function.
\item Taking advantage of the linearity, we may interchange the expectation with the linear functional, thereby reducing the optimal control problem to a problem of maximizing $l \mapsto \langle l,\, \E[ \X^{<\infty}_{0,t}] \rangle$ over some set of dual elements $l$.
\item Truncate the expected signature to a finite level $N$.
\end{enumerate}
The above strategy, in principle, only imposes very mild conditions on the underlying process $X$, mainly that it is continuous but possibly rough. In particular, $X$ does not need to be a Markov process or a semi-martingale. For this reason, we may consider the approach to be model-free.\footnote{The full expected signature $\E[ \X^{<\infty}]$ typically characterizes the distribution of the process $X$, see~\cite{CL16}. In that sense, we would hesitate to regard methods relying on the full (rather than truncated) expected signature as ``model-free''.}
Note, however, that the assumption of the existence of the expected signature $ \E [\X^{<\infty}_{0,t}]$ is a rather strong assumption -- in particular, ruling out many stochastic volatility models, such as the Heston model.

\cite{KLP20} contains extensive numerical examples, indicating the method's excellent performance in various scenarios and models, often beating benchmark methods from the financial literature. On the other hand, theoretical justification of the different approximation steps summarized above is largely missing.

We extend the method of \cite{KLP20} to another important control problem in finance, namely the \emph{optimal stopping} problem, or, in more financial terms, the pricing of \emph{American options}, for which we provide rigorous theory. 
More specifically, we are concerned with the problem of computing
\begin{equation}\label{eq:intro-optimal-stopping}
  \sup_{\tau \in \mathcal{S}} \E \left[ Y_{\tau \wedge T} \right],
\end{equation}
where $Y$ denotes a process adapted to the filtration $(\mathcal{F}_t)_{t \in [0,T]}$ generated by a rough path process $(X_t)_{t \in [0,T]}$, and $\mathcal{S}$ denotes the set of all stopping times w.r.t.~the same filtration. In a financial context, $Y$ usually denotes a reward process discounted with
respect to some num\'eraire. At first glance the optimal stopping problem~\eqref{eq:intro-optimal-stopping} may seem unsuitable for the signature-based approach, as typical candidate stopping times are hitting times of sets, which are generally discontinuous w.r.t.~the underlying path. We solve this issue by using \emph{randomized stopping times}, see \cite{BTW20}. Note that extending the set $\mathcal{S}$ to also include randomized stopping times does not change the value of~\eqref{eq:intro-optimal-stopping}. In the end, we are able to prove that replacing proper stopping times by \emph{signature stopping times} -- i.e., stopping times given in terms of linear functionals of the signature $\X^{<\infty}$ -- does not change the value of the optimal stopping problem either. More precisely, we have
\begin{theorem}
  \label{thr:intro-signature-stopping}
  Assume that $\E[ \| Y \|_\infty] < \infty$. Then,
  \begin{equation*}
    \sup_{\tau \in \mathcal{S}} \E[ Y_{\tau \wedge T}] = \sup_{\tau_l} \E [ Y_{\tau_l \wedge T}],
  \end{equation*}
  where the supremum on the right-hand-side ranges over stopping times $\tau_l \coloneqq \inf \{ t \in [0,T] \, : \, \langle l, \, \hat{\X}_{0,t}^{<\infty} \rangle \ge 1\}$ connected with linear functionals $l$ on the signature process $\hat{\X}_{0,t}^{<\infty}$ (we refer to Sections~\ref{sec:notat-basic-defin} and~\ref{sec:stopping-rules} for precise definitions).
\end{theorem}
The theorem is presented as Proposition~\ref{thm:signature-stopping} below. We note that, following \cite{KLP20}, we extend the path $X$ by adding running time as an additional component. $\hat{\X}^{<\infty}$ denotes the signature of the extended path.

In the next step we need to actually compute a maximizing signature stopping time.
In this context, this most importantly implies replacing the full signature $\hat{\X}^{<\infty}$ by a truncated version $\hat{\X}^{\le N}$. Using some further technical assumptions, Proposition~\ref{thr:main_approx} provides convergence of the corresponding approximations to the value of~\eqref{eq:intro-optimal-stopping} as $N \to \infty$.

Assuming that $Y$ is a polynomial function of $X$ -- or, more generally, of $\X^{<\infty}$
-- we can derive an approximation formula in terms of an optimization problem involving linear functionals of the expected signature $\E[ \hat{\X}^{\le N}]$ rather than the expectation of some functional of the signature. See Corollary~\ref{cor:opt-stopping-expected-sig} for details.

While this approach undoubtedly has many attractive features, very deep truncation levels $N$ might be needed for complicated optimal stopping problems. Hence, we also consider stopping rules which are parameterized by non-linear functionals of the (truncated) signature. More precisely, we consider stopping rules represented by \emph{deep neural networks} (DNNs) applied to the \emph{log-signature}. DNNs are well known for their combination of high expressiveness, especially in high dimensions, relative ease of training, and wide availability of efficient software implementations. On the other hand, the log-signature removes linear redundancies from the full signature, and, therefore, provides a compressed but equivalent representation of the data contained in the signature. We prove the following analogue to Theorem~\ref{thr:intro-signature-stopping}, see Proposition~\ref{prop:log_signature-stopping} for details of the construction and the proof.

\begin{theorem}
  \label{thr:intro-log-signature-stopping}
  Let $\mathcal{T}_{\mathrm{log}}$ denote the set of DNNs $\theta$ of a suitable chosen architecture, chosen such that $\theta$ accepts a truncated log-signature $\log^\otimes \widehat{\mathbb{X}}_{0,s}^{\le N}$ truncated at some level $N = N(\theta)$, $N=1, 2, \ldots$.
Define the \emph{deep, randomized signature stopping rule}
  \begin{equation*}
    \tau^r_{\theta} \coloneqq \inf \left\{ t \ge 0 \ : \ \int_0^t \theta\left( \log^\otimes \widehat{\mathbb{X}}_{0,s}^{\le N} \right)^2 ds \ge Z \right\},
    \end{equation*}
    for an independent r.v.~$Z \sim \mathrm{Exp}(1)$. Assuming that $\mathbb{E}\left[ \|Y\|_\infty \right] < \infty$, we then have
    \begin{equation*}
      \sup_{\tau \in \mathcal{S}} \mathbb{E}\left[ Y_{\tau \wedge T} \right] = \sup_{\theta \in \mathcal{T}_{\mathrm{log}}} \mathbb{E}\left[ Y_{\tau^r_\theta \wedge T} \right] .
    \end{equation*}
\end{theorem}

Finally, Section~\ref{sec:numerics} contains extensive numerical examples of both the linearized and the deep log signature method. 
We consider two models. 
Motivated by~\cite{becker2019deep} we consider the problem of optimally stopping a fractional Brownian motion. In this example, we find excellent performance of the deep signature stopping approach.
En passant, we provide explicit formulas for the limiting case $H\to0$.
We also consider an example from optimal control in an electricity market with the price function modeled by a fractional Brownian motion, as motivated by~\cite{bennedsen2017rough}.

\subsection*{Outline of the paper}

Section~\ref{sec:notat-basic-defin} recalls basic definitions from the theory of rough paths and provides the algebraic and analytic setting of signatures. A framework for studying stopped rough paths is presented in Section~\ref{sec:space-stopped-rough}. Stopping times based on continuous functionals on rough paths and their randomized counter-parts are introduced and discussed in Section~\ref{sec:stopping-rules}. The special case of linear functionals of the signature is studied in Section~\ref{sec:lin_stopping_rues}. The following Section~\ref{sec:appr-stopp-probl} contains a fully linear approximation to the optimal stopping problems in terms of the expected truncated signature. Deep signature stopping rules are introduced in Section~\ref{sec:non-linear-rules}. Finally, numerical examples are presented in Section~\ref{sec:numerics}.

\section{Preliminaries}
\label{sec:notat-basic-defin}

We start by introducing the basic definitions needed for understanding signatures and their algebraic and -- in the context of rough paths -- analytic properties. These definitions are standard in the rough path literature, we refer to \cite{LCL07, FH14,FV10} for a more detailed exposition.

\subsection{The tensor algebra}

 Let $V$ be a finite-dimensional $\R$-vector space with basis $\{e_1,\ldots,e_d\}$. The dual space is denoted by $V^*$ with dual basis $\{e_1^*,\ldots,e_d^*\}$. We define the \emph{tensor algebra} and the \emph{extended tensor algebra} by setting
 \begin{align*}
  T(V) \coloneqq \bigoplus_{n = 0}^{\infty} V^{\otimes n} \quad \text{and} \quad T((V)) \coloneqq \prod_{n = 0}^{\infty} V^{\otimes n}
 \end{align*}
 where $V^{\otimes n}$ denotes the $n$-th tensor power of $V$ with the convention $V^{\otimes 0} \coloneqq \R$, $V^{\otimes 1} \coloneqq V$. Note that there is a natural pairing between $T((V))$ and $T(V^*)$ which we denote by
 \begin{align*}
    \langle \cdot,\cdot \rangle \colon T(V^*) \times T((V)) \to \R.
 \end{align*}
 We define sum and product of two elements $\mathbf{a} = (a_n)_{n = 0}^{\infty}, \mathbf{b} = (b_n)_{n = 0}^{\infty} \in T((V))$ by setting
 \begin{align*}
  \mathbf{a} + \mathbf{b} &\coloneqq (a_n + b_n)_{n = 0}^{\infty}, \\
  \mathbf{a} \otimes \mathbf{b} &\coloneqq (\sum_{i + j = n} a_i \otimes b_j)_{n = 0}^{\infty}.
 \end{align*}
 For $\lambda \in \R$, we define $\lambda \mathbf{a} \coloneqq (\lambda a_n)_{n = 0}^{\infty}$. We also let  $\mathbf{0} \coloneqq (0,0,\ldots)$ and $\mathbf{1} \coloneqq (1,0,0,\ldots)$. Note that
 \begin{align*}
  \mathbf{1} \otimes \mathbf{a} = \mathbf{a} \otimes \mathbf{1} = \mathbf{a}
 \end{align*}
 for every $\mathbf{a} \in T((V))$. The \emph{truncated tensor algebra} is defined by
 \begin{align*}
  T^{N}(V) \coloneqq \bigoplus_{n = 0}^{N} V^{\otimes n}.
 \end{align*}
 We define maps $\pi_n \colon T((V)) \to V^{\otimes n}$ and $\pi_{\leq N} \colon T((V)) \to T^{N}(V)$ by $\pi_n(\mathbf{a}) = a_n$ and $\pi_{\leq N}(\mathbf{a}) = (a_0,\ldots,a_N)$ where $\mathbf{a} = (a_n)_{n = 0}^{\infty}$. We will sometimes abuse notation and write $\mathbf{0}$ and $\mathbf{1}$ for the elements $\pi_{\leq N}(\mathbf{0})$ and $\pi_{\leq N}(\mathbf{1})$ in the truncated tensor algebra.

 Next, we consider norms on $T((V))$ and $T(V^*)$. On $V$, we choose the $l^{\infty}$-norm, i.e. for $v = \lambda_1 e_1 + \ldots + \lambda_d e_d$, we set $|v| := \max_i |\lambda_i|$. For elements in $V^*$, we use the $l^1$-norm, i.e. $|v^*| := |\lambda_1| + \ldots + |\lambda_d|$ for $v^* = \lambda_1 e^*_1 + \ldots + \lambda_d e^*_d$. On the tensor powers of $V$ resp. $V^*$, we use the corresponding norms, too. Note that the norms on the tensor products $V^{\otimes n}$ are \emph{admissible}, meaning that if $v = a_1 \otimes \ldots \otimes a_k$ and $\sigma v \coloneqq a_{\sigma(1)} \otimes \ldots \otimes a_{\sigma(k)}$ for a permutation $\sigma$, $|\sigma v| = |v|$, and $|v \otimes w| \leq |v||w|$.
We set
\begin{align*}
 |\mathbf{a}| \coloneqq \sup_{i \in \N_0} |a_i| \in [0,\infty] \quad \text{for } \mathbf{a} = (a_i)_{i = 0}^{\infty} \in T((V))
\end{align*}
and
\begin{align*}
 |\mathbf{b}| \coloneqq \sum_{i = 0}^{\infty} |b_i| \in [0,\infty) \quad \text{for } \mathbf{b} = (b_i)_{i = 0}^{\infty} \in T(V^*).
\end{align*}
Note that we always have
\begin{align*}
 |\langle \mathbf{b}, \mathbf{a} \rangle| \leq |\mathbf{b}||\pi_{\leq N} (\mathbf{a})| \leq |\mathbf{b}||\mathbf{a}|
\end{align*}
where $N = \max\{i \in \N_0\, :\, b_i \neq 0 \}$.

 \subsection{Shuffles}

 In the following, calculations will mainly be performed in the space $T(V^\ast)$. In order to simplify notations, we will replace expressions like $e_{i_1}^* \otimes \cdots \otimes e_{i_n}^*$ by the much simpler form $\blue{i_1} \cdots \blue{i_n}$.
 More precisely, let $\mathcal{W}(\mathcal{A}_d)$ be the linear span of words composed by the letters in the dictionary $\mathcal{A}_d = \{\blue{1},\ldots,\blue{d} \}$. The empty word is denoted by ${\varnothing} \in \mathcal{W}(\mathcal{A}_d)$. We can naturally define the sum $l_1 + l_2$ and the scalar product $\lambda l$ for elements $l,l_1,l_2 \in \mathcal{W}(\mathcal{A}_d)$ and $\lambda \in \R$. If ${w} = \blue{i_1} \cdots \blue{i_n}$ and ${v} = \blue{j_1} \cdots \blue{j_m}$ are two words, the \emph{concatenation} is defined by
 \begin{align*}
  {w} {v} \coloneqq \blue{i_1} \cdots \blue{i_n} \blue{j_1} \cdots \blue{j_m}.
 \end{align*}
 This operation is extended bi-linearly to elements in $\mathcal{W}(\mathcal{A}_d)$. The basis elements $\{ e_{i_1}^* \otimes \cdots \otimes e_{i_n}^*\, :\, i_1,\ldots,i_n \in \{1,\ldots,d\}\}$ in $(V^*)^{\otimes n}$ can be identified with words via the map
 \begin{align*}
  e_{i_1}^* \otimes \cdots \otimes e_{i_n}^* \mapsto \blue{i_1} \cdots \blue{i_n}
 \end{align*}
 which induces an isomorphism $T(V^*) \cong \mathcal{W}(\mathcal{A}_d)$.  We can also think of $\mathcal{W}(\mathcal{A}_d)$ as the space of non-commutative polynomials where the unknown are given by the letters $\{\blue{1},\ldots,\blue{d} \}$. For a word ${w} = \blue{i_1} \cdots \blue{i_n}$, set $\operatorname{deg}(w) \coloneqq n$ and $\operatorname{deg}(\varnothing) \coloneqq 0$. If $l = \lambda_1 {w_1} + \ldots + \lambda_n {w_n} \in \mathcal{W}(\mathcal{A}_d)$ with $\lambda_1, \ldots, \lambda_n \in \R\setminus\{0\}$ and $w_1, \ldots, w_n$ words, we define
 \begin{align*}
  \operatorname{deg}(l) \coloneqq \max_{i = 1\ldots,n} \operatorname{deg}(w_i).
 \end{align*}
 
 Apart from concatenation, there is a second important product defined on $\mathcal{W}(\mathcal{A}_d)$ which is called \emph{shuffle product}: For a word $w$, we set
 \begin{align*}
  {w} \shuffle {\varnothing} \coloneqq  {\varnothing} \shuffle {w} \coloneqq {w}.
 \end{align*}
 If ${w}\blue{i}$ and ${v}\blue{j}$ are words and $\blue{i},\blue{j} \in \mathcal{A}_d$ are letters, we recursively define  $w\blue{i} \shuffle v\blue{j} \in \mathcal{W}(\mathcal{A}_d)$ by
 \begin{align*}
  w\blue{i} \shuffle v\blue{j} \coloneqq (w \shuffle v\blue{j})\blue{i} + (w\blue{i} \shuffle v)\blue{j}.
 \end{align*}
 This operation is extended bi-linearly to a product $\shuffle \colon \mathcal{W}(\mathcal{A}_d) \times \mathcal{W}(\mathcal{A}_d) \to \mathcal{W}(\mathcal{A}_d)$. Note that $\shuffle$ is associative, commutative and distributive over $+$. If $P \in \R[x]$ is a commutative polynomial with unknown variable $x$, i.e. $P(x) = \lambda_0 + \lambda_1 x + \ldots \lambda_n x^n$, we define $P^{\shuffle} \colon \mathcal{W}(\mathcal{A}_d) \to \mathcal{W}(\mathcal{A}_d)$ by setting
 \begin{equation}\label{eq:polynomial-shuffle}
  P^{\shuffle}(l) \coloneqq \lambda_0 {\varnothing} + \lambda_1 l + \lambda_2 (l \shuffle l) + \ldots + \lambda_n l^{\shuffle n},
 \end{equation}
 where $l^{\shuffle k}$ denotes $k$-th shuffle product of $l \in \mathcal{W}(\mathcal{A}_d)$ with itself.

 We define
 \begin{equation}\label{eq:group}
  G(V) \coloneqq \{\mathbf{a} \in T((V)) \setminus\{\mathbf{0}\} \, :\, \langle l_1 \shuffle l_2, \mathbf{a} \rangle = \langle l_1, \mathbf{a} \rangle \langle  l_2, \mathbf{a} \rangle \text{ for every } l_1, l_2 \in T(V^*)  \}
 \end{equation}
 and call it the set of \emph{group-like elements}. Note that $\pi_0(\mathbf{g}) = 1$ for every $\mathbf{g} \in G(V)$. One can show that $(G(V), \otimes)$ is a group with identity $\mathbf{1}$ and inverse given by
 \begin{align*}
  \mathbf{g}^{-1} = \sum_{n \geq 0} (\mathbf{1} - \mathbf{g})^{\otimes n}.
 \end{align*}
 We also set $G^N(V) \coloneqq \pi_{\leq N}(G(V))$ which is a free nilpotent group of order $N$ with respect to the
 ``truncated multiplication'' $\mathbf{a}\,{\otimes}_{G^N(V)}\,\mathbf{b}$ $:=$ $\pi_N(\mathbf{a}\otimes\mathbf{b}),$  for $\mathbf{a},\mathbf{b}\in G^N(V)$. However, we will not distinguish between the multiplication symbols on $G^N(V)$ and $G(V)$ and use $\otimes$ in both cases.

 \begin{remark}
   The relation $\langle l_1 \shuffle l_2, \mathbf{a} \rangle = \langle l_1, \mathbf{a} \rangle \langle  l_2, \mathbf{a} \rangle$ for $\mathbf{a} \in G(V)$ implies that
   \begin{align}\label{eqn:lin_sig}
    P(\langle l,  \mathbf{a} \rangle) = \langle P^{\shuffle}(l), \mathbf{a} \rangle
   \end{align}
   for any polynomial $P$. This is really the justification for introducing the shuffle product, as it provides an explicit linearization of polynomials in the signature.
 \end{remark}

 \subsection{Rough paths and their signatures}

Now that the algebraic setting for signatures is developed (for the purposes of this paper), we can finally consider the analytic properties of (rough) paths. More concretely, given a path $X:[0,T] \to V$ (of sufficient regularity), we will associate to it a function $\mathbb{X}$ taking values in the truncated tensor algebra, which is the fundamental building block of rough path theory.
Set $\Delta_T \coloneqq \{(s,t) \in [0,T]^2 \, :\, 0 \leq s \leq t \leq T\}$. For a map $\mathbb{X} \colon \Delta_T \to T^{N}(V)$, we define its \emph{$p$-variation}
\begin{align*}
 \| \mathbb{X} \|_{p-\mathrm{var};[s,t]} \coloneqq \max_{k = 1,\ldots,N} \sup_{\mathcal{D} \subset [s,t]} \left( \sum_{t_i \in \mathcal{D}} |\pi_k(\mathbb{X}_{t_i,t_{i+1}})|^{\frac{p}{k}} \right)^{\frac{k}{p}}
\end{align*}
where the supremum ranges over all partitions $\mathcal{D}$ of $[s,t]$. We will use the notation $\| \mathbb{X} \|_{p-\mathrm{var}} \coloneqq \| \mathbb{X} \|_{p-\mathrm{var};[0,T]}$. For $\mathbb{X}, \Y \colon \Delta_T \to T^{N}(V)$, we define the \emph{$p$-variation distance}
\begin{align*}
 d_{p-\mathrm{var};[s,t]}(\X,\Y) \coloneqq \| \mathbb{X} - \Y\|_{p-\mathrm{var};[s,t]}
\end{align*}
and set $d_{p-\mathrm{var}}(\X,\Y) \coloneqq d_{p-\mathrm{var};[0,T]}(\X,\Y)$. A \emph{weakly geometric $p$-rough path $\mathbb{X}$} is a continuous path $\mathbb{X} \colon [0,T] \to G^{\lfloor p \rfloor}(V)$ with $\X_0 = \mathbf{1}$ and $\|\X\|_{p-\mathrm{var}} < \infty$ where we set $\X_{s,t} \coloneqq \X_s^{-1} \otimes \X_t$ for $s \leq t$. Note that $\X_t = \X_{0,t}$. We denote the space of weakly geometric $p$-rough paths by $\mathcal{W} \Omega_T^p$ and equip it with the distance $d_{p-\mathrm{var}}$. If $X \colon [0,T] \to V$ is a continuous path of bounded variation, we define its \emph{signature} $\X^{<\infty} \colon [0,T] \to T((V))$ by
\begin{align*}
 \pi_k(\X^{<\infty}_t) \coloneqq \int_{0 < t_1 < \ldots < t_k < t} dX_{t_1} \otimes \cdots \otimes dX_{t_k}.
\end{align*}
The \emph{truncated signature} $\X^{\leq N} \colon [0,T] \to T^N(V)$ is defined by $\X^{\leq N} \coloneqq \pi_{\leq N}(\X^{<\infty})$. It can be checked that $\X^{<\infty}$ takes values in $G(V)$ and we set $\X^{<\infty}_{s,t} \coloneqq (\X^{<\infty}_s)^{-1} \otimes \X^{<\infty}_t$ so that
\begin{align*}
 \pi_k(\X^{<\infty}_{s,t}) = \int_{s < t_1 < \ldots < t_k < t} dX_{t_1} \otimes \cdots \otimes dX_{t_k}.
\end{align*}
One can also show that $\X^{\leq N}$ is an element in $\mathcal{W} \Omega_T^p$ for every $p \geq 1$ with $N = \lfloor p \rfloor$.

A \emph{geometric $p$-rough path $\mathbb{X}$} is a weakly geometric rough path $\X \in \mathcal{W} \Omega_T^p$ for which there exists a sequence of piecewise smooth paths $(X_n)$ such that $d_{p-\mathrm{var}}(\X, \X_n^{\leq \lfloor p \rfloor}) \to 0$ as $n \to \infty$. The space of geometric rough paths is denoted by $\Omega_T^p$. It can be shown that the inclusion $\Omega_T^p \subset \mathcal{W} \Omega_T^p$ is strict and that $\Omega_T^p$ is a Polish space. From Lyons' Extension theorem \cite[Theorem 3.7]{LCL07}, every geometric rough path $\mathbb{X} \in \Omega_T^p$ has a unique lift $\mathbb{X}^{<\infty}$ which is a path in $G(V)$, satisfying $\| \pi_{\leq N}(\mathbb{X}^{<\infty}) \|_{p-\mathrm{var}} < \infty$ for every $N \geq 1$ and $\pi_{\leq \lfloor p \rfloor}(\mathbb{X}^{<\infty}) = \mathbb{X}$. We call $\mathbb{X}^{<\infty}$ the \emph{signature of the rough path $\mathbb{X}$} and $\mathbb{X}^{\leq N} \coloneqq \pi_{\leq N}(\mathbb{X}^{<\infty})$ its \emph{truncated signature}.

Similarly, for $V = \R^{1+d}$, we define the space $\hat{\Omega}_T^p$ as the closure of rough path lifts $\hat{\X}^{\leq \lfloor p \rfloor}$ in the $p$-variation distance where $\hat{X}_t = (t,X_t) \in \R^{1+d}$ and $X$ is piecewise smooth. It follows that $\hat{\Omega}_T^p$ is Polish.

\begin{remark}
  \label{rem:path-letters}
  Following the notation introduced above, the letter $\blue{1}$ corresponds to the running time component $t$ of the path $\hat{X}$, whereas the components of $X$ correspond to the letters $\blue{2}, \ldots, \blue{d+1}$, respectively.
\end{remark}

\begin{example}[Brownian motion as a rough path]\label{exmpl:brownian-motion} Let $X$ be a $d$-dimensional Brownian motion.
In this case a natural lift to a geometric rough path $\mathbb{X}\in\Omega^{p}_T$ with $p\in(2, 3)$ is given by
\begin{align*}
\mathbb{X}_{s,t} = \left(1, X_{s,t}, \int_s^t X_{s, u} \otimes \circ dX_u \right), \quad 0 \le s \le t \le T
\end{align*}
where $X_{s,t} = X_t - X_s$ and for all $\blue{i},\blue{j} \in \mathcal{A}_d$ the tensor valued Stratonovich integral is given by
\begin{align*}
\left\langle \blue{ij}, \int_s^t X_{s, u} \otimes \circ dX_u\right\rangle = \int_s^t X^{i}_{s,u} \circ d X^{j}_u = \int_s^t X^{i}_{s,u} d X^{j}_u + \frac{1}{2}[X^i, X^j]_{s,t}.
\end{align*}
Indeed, to see that $\mathbb{X}\in\mathcal{W}\Omega_T^p$, one may readily check that $\mathbb{X}_{s,t} \in G^2(V)$ is an immediate consequence of the product rule and the rough path regularity of $\mathbb{X}$ follows from a generalized Kolmogorov criterion (see \cite[Theorem 3.1]{FH14}).
As it is well known that the integral with respect to the piecewise linear approximation of Brownian motion converges to the Stratonovich integral a.s., we also see that $\mathbb{X}\in\Omega_T^p$.
The signature $\mathbb{X}^{<\infty}$ of the enhanced Brownian motion $\mathbb{X}$ is then given by the iterated integrals of all orders, i.e. for any word $w= \blue{i_1}\cdots\blue{i_k} \in \mathcal{W}(\mathcal{A}_d)$ we have
\begin{align*}
\langle \blue{i_1}\cdots\blue{i_k}, \X^{<\infty}_{0,t} \rangle = \int_{0 < t_1 < \ldots < t_k < t} \circ dX^{i_1}_{t_1} \cdots \circ dX^{i_k}_{t_k}.
\end{align*}
An explicit form of the expected signature is also known due to Fawcett \cite{fawcett2002problems}
\begin{align*}
\E(\X^{<\infty}_{0,t}) = \exp^{\otimes}\left(\frac{1}{2}t\sum_{i=1}^d e_i \otimes e_i \right) = \sum_{n=0}^\infty \frac{1}{n!}\frac{t^n}{2^n} \left(\sum_{i=1}^d e_i \otimes e_i \right)^{\otimes n}.
\end{align*}
Note that this construction of a geometric rough path works in principle for all continuous semimartingales and we refer to \cite[Section 14]{FV10} for more detail.
\end{example}

\begin{example}[Fractional Brownian motion]\label{exmpl:fbm}
Let $X$ be a one-dimensional fractional Brownian motion with Hurst parameter $H\in (0,1)$, i.e. $X$ is a zero mean Gaussian process with covariance function
\begin{align}\label{eq:cov_fbm}
\E( X_s X_t ) = \frac{1}{2}\left( |s|^{2H} + |t|^{2H} - |t-s|^{2H}\right), \quad 0 \le s \le t.
\end{align}
Recall that the sample paths of $X$ are a.s.~$(H-\varepsilon)$-H\"older continuous for any $\varepsilon > 0$. In case $H=1/2$ $X$ is just a standard Brownian motion, in case $H \neq 1/2$, however, $X$ is not a Markov process and not a semimartingale.
However, since $X$ is one-dimensional ($V=\R$) there is a trivial lift to a geometric rough path $\mathbb{X}\in\Omega_T^{p}$ for any $p\in(1/H, 1+ 1/H)$ given by
\begin{align*}
\mathbb{X}_{s,t} = \left(1, X_{s,t}, \frac{1}{2}(X_{s,t})^{\otimes 2}, ..., \frac{1}{\lfloor p\rfloor!}(X_{s,t})^{\otimes \lfloor p\rfloor} \right) \equiv \texp_{\lfloor p\rfloor}(X_{s,t }) \in G^{\lfloor p\rfloor}(V), \quad 0 \le s \le t \le T.
\end{align*}
As we will see in the next section, we are particulary interested in the process $\hat{X}$ defined by $\hat{X}_t = (t,X_t)\in \R^2$.
The first component of $\hat{X}$ is of locally bounded variation and therefore it can be lifted to a geometric rough path $\hat{\mathbb{X}}\in\hat\Omega_T^{p}$ (see \cite[Theorem 9.26]{FV10}).
Intuitively speaking we can make use of the abundant regularity of the first component $\hat{X}$ in order to define iterated integrals by imposing the integration by parts rule.
More precisely in case $p > 2$ we have
\begin{align*}
\langle \blue{12}, \mathbb{\hat{\mathbb{X}}}_{s,t}\rangle =  \langle \blue{2}, \mathbb{\hat{\mathbb{X}}}_{s,t}\rangle\langle \blue{1}, \mathbb{\hat{\mathbb{X}}}_{s,t}\rangle - \langle \blue{21}, \mathbb{\hat{\mathbb{X}}}_{s,t}\rangle= X_{s,t}(t-s) - \int_{s}^t X_{s,u} du,
\end{align*}
and the right hand side is clearly well defined.
Using the shuffle identity, this reasoning can be carried on to express all components of the signature $\hat{\mathbb{X}}^{<\infty}$ in terms of increments of $X$, finite-variation integrals and products thereof.
\end{example}

\section{The space of stopped rough paths}
\label{sec:space-stopped-rough}

We will now consider rough paths $\Z$ defined on some intervals $[0,s] \subset [0,T]$. In order to naturally model the notion of \emph{adaptedness} to a filtration, we will consider functionals of the restriction of a rough path $\Z$ to a subinterval of its domain. Hence, the analysis of the corresponding control problem requires us to define a distance of rough paths with different domains. Following~\cite{KLP20}, we will use a distance motivated by Dupire's functional It\=o calculus, see \cite{Dup19, ContFournie10}. This means, when we compare a path $Z^1$ defined on $[0,s]$ and another path $Z^2$ defined on $[0,t]$ with $s<t$, we will extend $Z^1$ to $[0,t]$ by $Z^1_u \coloneqq Z^1_s$ for $s \le u \le t$. We will, in principle, use the same construction for rough paths, but recall that we are considering paths $u \mapsto (u,X_u)$ in our framework, and extending the time component of such a path in a constant way does not make much sense. Instead, we will apply Dupire's extension to the $X$-component, but use the linear extension (i.e., the exact one) for the time component.

More precisely, let $\Z|_{[0,s]} \in \hat{\Omega}_s^p$ and $s \leq t$. By definition, there exists a sequence $Z^n_u = (u,X^n_u)$ where $X^n \colon [0,s] \to \R^d$ is a piecewise smooth path such that $d_{p-\mathrm{var};[0,s]}(\Z|_{[0,s]},\Z^{n; \leq \lfloor p \rfloor}) \to 0$ as $n \to \infty$. Set $\tilde{X}^n_u \coloneqq X^n_{u \wedge s}$ for $u \in [0,t]$ and $\tilde{Z}^n_u \coloneqq (u, \tilde{X}^n_u)$. One can check that $\tilde{\mathbb{Z}}^{n; \leq \lfloor p \rfloor}$ is a Cauchy sequence in $\hat{\Omega}^p_t$, and we denote the limit by $\tilde{\Z}|_{[0,t]}$. One can also check that the definition of $\tilde{\Z}|_{[0,t]}$ does not depend on the choice of the sequence $X^n$. By construction we have that $\tilde{\Z}|_{[0,s]} = \Z|_{[0,s]}$, which motivates the following definition.

\begin{definition}\label{defn:stopped_rp}
  For $T > 0$, we set $\Lambda _T \coloneqq \bigcup_{t \in [0,T]} \hat{{\Omega}}_t^p$ and call it the \emph{space of stopped rough paths}. We equip it with the metric
  \begin{align*}
    d(\mathbb{X}|_{[0,t]}, \mathbb{Y}|_{[0,s]}) \coloneqq d_{p-\mathrm{var};[0,t]}(\mathbb{X}|_{[0,t]}, \tilde{\mathbb{Y}}|_{[0,t]}) + |t-s|
  \end{align*}
  where we assume $s \leq t$ and $\tilde{\mathbb{Y}}|_{[0,t]}$ is the stopped rough path constructed as explained above.
\end{definition}

Let us mention that $\Lambda_T$ is Polish. For this and related simple technical facts about the topology of $\Lambda_T$, we refer to the Appendix \ref{sec:techn-aspects-stopp}.
Later on we will use that $\mathbbm{1}_{\{ \tau(\omega) \leq t \}}$ can be represented as a measurable map of the restricted rough path.

\begin{lemma}\label{lem:ex_theta_stop}
Let $\hat{\mathbb{X}}$ be a stochastic process in $\hat{\Omega}_T^p$ and set $\mathcal{F}_t \coloneqq \sigma(\hat{\mathbb{X}}_{0,s}\, :\, 0 \leq s \leq t) = \sigma(\hat{\X}|_{[0,t]})$. Let $\tau$ be a stopping time with respect to $(\mathcal{F}_t)$. Then there is a Borel measurable map $\theta \colon \Lambda_T \to \{0,1\}$ such that
\begin{align*}
   \theta({\hat{\mathbb{X}}}(\omega)|_{[0,t]}) = \mathbbm{1}_{\{ \tau(\omega) \leq t \}}
\end{align*}
for every $\omega \in \Omega$.
\end{lemma}

 \begin{proof}
For every $t \in [0,T]$, $\{\tau \leq t\}$ is $\sigma(\hat{\mathbb{X}}|_{[0,t]})$-measurable, hence there is a set $A_t \in \mathcal{B}({\hat{\Omega}}^p_t)$ such that $(\hat{\mathbb{X}}|_{[0,t]})^{-1}(A_t) = \{\tau \leq t\}$. It follows that
  \begin{align}\label{eqn:setsA}
    \mathbbm{1}_{\{ \tau(\omega) \leq t \}} = \mathbbm{1}_{A_t}({\hat{\mathbb{X}}}(\omega)|_{[0,t]})
  \end{align}
  for every $\omega \in \Omega$. Define $\phi \colon \Lambda_T \to [0,T] \times \hat{\Omega}_T^p$ as $\phi(\mathbb{X}|_{[0,t]}) = (t,\tilde{\mathbb{X}}|_{[0,T]})$ where $\tilde{\mathbb{X}}|_{[0,T]}$ denotes the stopped process defined in Definition \ref{defn:stopped_rp} . Note that $\phi$ is continuous, thus measurable. Define $f \colon [0,T] \times \hat{\Omega}_T^p \to \R$ as $f(t,\hat{\mathbb{X}}) \coloneqq \mathbbm{1}_{A_t}(\hat{\mathbb{X}}|_{[0,t]})$. For fixed $t$, $\hat{\mathbb{X}} \to \hat{\mathbb{X}}|_{[0,t]}$ is continuous and $\hat{\mathbb{X}}|_{[0,t]} \mapsto \mathbbm{1}_{A_t}(\hat{\mathbb{X}}|_{[0,t]})$ is measurable, therefore $\hat{\mathbb{X}} \mapsto f(t,\hat{\mathbb{X}})$ is measurable. For $n \in \N$, define $I^n_k \coloneqq [k/2^n T, (k+1)/2^n T)$ for $k = 0, \ldots, 2^n - 2$, $I^n_{2^n -1} \coloneqq [(2^n -1)/2^n T, T]$ and $t^n_k \coloneqq k/2^n T$. Set
  \begin{align*}
   f_n(t,\hat{\mathbb{X}}) \coloneqq \sum_{k = 0}^{2^n -1} f(t^n_k,\hat{\mathbb{X}}) \mathbbm{1}_{I^n_k}(t)
  \end{align*}
  which is measurable for every $n \in \N$. Set
  \begin{align*}
    \tilde{f}(t,\hat{\mathbb{X}}) \coloneqq \limsup_{m \to \infty} \limsup_{n \to \infty} f_n(t+1/m,\hat{\mathbb{X}}) \quad \text{and} \quad \theta(\hat{\mathbb{X}}|_{[0,t]}) \coloneqq (\tilde{f} \circ \phi)(\hat{\mathbb{X}}|_{[0,t]}).
  \end{align*}
  The map $\theta$ is thus measurable and satisfies
  \begin{equation*}
   \theta(\hat{\mathbb{X}}(\omega)|_{[0,t]}) = \limsup_{m \to \infty} \limsup_{n \to \infty} \sum_{k = 0}^{2^n -1} \mathbbm{1}_{\{ \tau(\omega) \leq t^n_k \}}\mathbbm{1}_{I^n_k}(t + 1/m) = \mathbbm{1}_{\{ \tau(\omega) \leq t \}}. \qedhere
  \end{equation*}
 \end{proof}

\section{Randomized stopping times}
\label{sec:stopping-rules}
In Lemma \ref{lem:ex_theta_stop} we have seen that a stopping time can be represented
by a stopping policy of the form $\theta: \Lambda_T \to \{0, 1\}$.
In the next step, we reverse the order and define stopping times associated to \emph{continuous} stopping policies of the form $\theta: \Lambda_T \to \R$.
Of course, for a given stopping time, there is no reason why it should be representable by a continuous stopping policy.
Indeed, relevant stopping times -- such as hitting times of even nice sets -- are often discontinuous functions of the underlying path. 
We will see, however, that stopping times, in particular the optimal stopping times for our problem, can be approximated by stopping times induced by continuous policies, in the sense that the corresponding value functions converge.
Later, in Section \ref{sec:lin_stopping_rues} and Section \ref{sec:non-linear-rules} we will see that this is also true when we restrict to the subclasses of continuous stopping policies given by linear functionals respectively deep neural networks applied to the signature.

Let $(\Omega,\mathcal{F},\P)$ be a probability space. In this and the following sections, $\hat{\mathbb{X}}$ denotes a stochastic process in $\hat{\Omega}^p_T$ and $Y \colon [0,T] \times \Omega \to \R$ is a real-valued continuous stochastic process adapted to the filtration $(\mathcal{F}_t)$, $\mathcal{F}_t = \sigma(\hat{\X}_{0,s}\, :\, 0 \leq s \leq t)$. 
Denote by $\mathcal{S}$ the space of all $(\mathcal{F}_t)$-stopping times.
We are trying to solve the optimal stopping problem for $Y$, i.e., in a financial context $Y$ corresponds to a cash-flow process.
For simplicity, we assume that $X_0 = 0$.

We consider \emph{randomized stopping times}, which relax proper stopping times and lead to much more regular approximation problems. We note that similar techniques have been used in \cite{BTW20} in the context of numerical methods for American option pricing.

\begin{definition}\label{def:rand-stopping-times}
We set $\mathcal{T} \coloneqq C(\Lambda_T,\R)$ and call it the space of \emph{continuous stopping policies}. 
Let $Z$ be a non-negative random variable independent of $\hat{\mathbb{X}}$ and such that $\P(Z = 0) = 0$. For a continuous stopping policy $\theta \in \mathcal{T}$, we define the \emph{randomized stopping time} by
    \begin{align}\label{eqn:randomized_stopping_time}
      \tau^r_{\theta} \coloneqq \inf \left\{t \geq 0\,:\, \int_0^{t \wedge T} \theta (\hat{\mathbb{X}}|_{[0,s]})^2\, ds \geq Z \right\}
    \end{align}
    where $\inf \emptyset = + \infty$.
 \end{definition}

Next we prove that stopping times can be approximated by randomized stopping times based on continuous stopping policies.

\begin{proposition}\label{prop:opt_cont_pol}
 For every stopping time $\tau \in \mathcal{S}$, there exists a sequence $\theta_n \in \mathcal{T}$ such that the randomized stopping times $\tau^r_{\theta_n}$ satisfy $\tau_{\theta_n}^r \to \tau$ almost surely as $n \to \infty$. In particular, if $\E[ \|Y\|_{\infty}] < \infty$, then
 \begin{align*}
  \sup_{\theta \in \mathcal{T}} \E[{Y}_{\tau^r_{\theta} \wedge T}] = \sup_{\tau \in \mathcal{S}} \E[Y_{\tau \wedge T}].
 \end{align*}
\end{proposition}

\begin{proof}
 Let $\tau$ be a stopping time. From Lemma \ref{lem:ex_theta_stop}, we know that there is a measurable map $\theta \colon \Lambda_T \to \{0,1\}$ such that
 \begin{align*}
  \theta(\hat{\mathbb{X}}|_{[0,t]}) = \mathbbm{1}_{\{\tau \leq t\}}.
 \end{align*}
Using \cite[Theorem 1]{Wis94}, we can find a sequence of continuous functions $\tilde{\theta}_n \in \mathcal{T}$ such that $\tilde{\theta}_n(\hat{\mathbb{X}}|_{[0,t]}) \to \mathbbm{1}_{\{\tau \leq t\}}$ almost surely w.r.t. $\lambda|_{[0,T]} \otimes \P$ where $\lambda|_{[0,T]}$ denotes the Lebesgue measure on $[0,T]$. W.l.o.g, we may assume that $0 \leq \tilde{\theta}_n \leq 1$. Set $\theta_n \coloneqq  (2\tilde{\theta}_n)^n$. Then
 \begin{align*}
  \lim_{n \to \infty} \theta_n(\hat{\mathbb{X}}|_{[0,t]}) \to \begin{cases}
                + \infty &\text{if } t \geq \tau \\
                0 &\text{if } t < \tau.
               \end{cases}
 \end{align*}
  It follows that $\tau_{\theta_n}^r \to \tau$ almost surely as $n \to \infty$. Using the dominated convergence theorem, this implies that
 \begin{align*}
  \sup_{\theta \in \mathcal{T}} \E[{Y}_{\tau^r_{\theta} \wedge T}] \geq \sup_{\tau \in \mathcal{S}} \E[Y_{\tau \wedge T}].
 \end{align*}
   To show the converse inequality, take $\theta \in \mathcal{T}$. From independence,
  \begin{align*}
   \E[ {Y}_{\tau_{\theta}^r \wedge T}\, |\, \hat{\mathbb{X}}] = \int_0^{\infty} {Y}_{\tau_z \wedge T}\, \P_Z(dz)
  \end{align*}
  where
  \begin{align*}
   \tau_z \coloneqq \inf \left\{t \geq 0\,:\, \int_0^{t \wedge T} \theta(\hat{\mathbb{X}}|_{[0,s]})^2\, ds \geq z \right\}.
  \end{align*}
  Note that this is a stopping time for every $z \geq 0$. Taking expectation, it follows that
  \begin{align*}
   \E[ Y_{\tau_{\theta}^r \wedge T}] = \int_0^{\infty} \E[{Y}_{\tau_z \wedge T}] \, \P_Z(dz) \leq \sup_{\tau \in \mathcal{S}} \E[{Y}_{\tau \wedge T}]
  \end{align*}
  which implies the claim.
\end{proof}

Note that we cannot generally assume that $\theta_n \to \theta$ implies $\tau_{\theta_n}^r \to \tau_{\theta}^r$, as is shown by the following counter-example. So even randomized stopping times are not continuous w.r.t.~the underlying stopping policies.
\begin{example}
Consider $\vartheta, \vartheta_n \colon [0,3] \to [0,\infty)$ defined by
\begin{align*}
 \vartheta(t) = \begin{cases}
                 1- t &\text{ if } t \in [0,1] \\
                 0    &\text{ if } t \in [1,2] \\
                 t - 2 &\text{ if } t \in [2,3]
                \end{cases}
                \quad \text{and} \quad
 \vartheta_n(t) = \begin{cases}
                 (1 - \frac{1}{n}) (1- t) &\text{ if } t \in [0,1] \\
                 0    &\text{ if } t \in [1,2] \\
                 t - 2 &\text{ if } t \in [2,3].
                \end{cases}
\end{align*}
Although $\vartheta_n \to \vartheta$ as $n \to \infty$, we have
\begin{align*}
 \inf \left\{t \geq 0\,:\, \int_0^{t \wedge 3} \vartheta(s)\, ds \geq \frac{1}{2} \right\} = 1 \quad \text{and} \quad \inf \left\{t \geq 0\,:\, \int_0^{t \wedge 3} \vartheta_n(s)\, ds \geq \frac{1}{2} \right\} > 2
\end{align*}
for all $n  \geq 1$.
\end{example}

As mentioned above, randomized stopping times \emph{regularize} the optimal stopping problem. Indeed, given a randomized stopping time $\tau^r_\theta$ defined in terms of an independent random variable $Z$ as in Definition~\ref{def:rand-stopping-times}, if we integrate the stopped process $Y_{\tau^r_\theta \wedge T}$ w.r.t.~$Z$, we obtain a smooth function of $\theta$ -- which is clearly not true without regularization, see Remark~\ref{rem:non-randomized-representation} below.

\begin{proposition}\label{lem:regularized_value}
 Let $S$ be an $(\mathcal{F}_t)$-stopping time and let $F_Z$ denote the cumulative distribution function of $Z$. Then
 \begin{align*}
  \E[{Y}_{\tau_{\theta}^r \wedge S} \, |\, \hat{\mathbb{X}}] = \int_0^S{Y}_{t}\, d \tilde{F}(t) + {Y}_S (1 - \tilde{F}(S)) = \int_0^S (1 - \tilde{F}(t))\, d{Y}_t + {Y}_0
 \end{align*}
 where the second integral is implicitly defined by integration by parts and
 \begin{align*}
  \tilde{F}(t) \coloneqq F_Z \left( \int_0^{t} \theta (\hat{\mathbb{X}}|_{[0,s]})^2 \, ds \right).
 \end{align*}
 In particular, if $Z$ has a density $\varrho$,
 \begin{align*}
  \E[{Y}_{\tau_{\theta}^r \wedge S}] = \E \left[  \int_0^S {Y}_{t} \theta (\hat{\mathbb{X}}|_{[0,t]})^2 \varrho \left( \int_0^{t} \theta (\hat{\mathbb{X}}|_{[0,s]})^2 \, ds \right)\, dt + {Y}_S (1 - \tilde{F}(S)) \right].
 \end{align*}
\end{proposition}

\begin{proof}
 Recall that $\tau_{\theta}^r \in [0,T] \cup \{\infty\}$. For $t \in [0,\infty)$, we have
 \begin{align*}
  \P(\tau_{\theta}^r \leq t\, |\, \hat{\mathbb{X}}) = \P \left( \int_0^{t \wedge T} \theta (\hat{\mathbb{X}}|_{[0,s]})^2 \, ds \geq Z\, |\, \hat{\mathbb{X}} \right) = F_Z \left( \int_0^{t \wedge T} \theta (\hat{\mathbb{X}}|_{[0,s]})^2 \, ds \right) = \tilde{F}(t)
 \end{align*}
 and
 \begin{align*}
  \P(\tau_{\theta}^r = \infty \, |\, \hat{\mathbb{X}}) =  \P \left( \int_0^T  \theta (\hat{\mathbb{X}}|_{[0,s]})^2 \, ds < Z\, |\, \hat{\mathbb{X}} \right) = 1 - \tilde{F}(T).
 \end{align*}
 It follows that for $f \colon [0,\infty] \to \R$ integrable,
 \begin{align*}
  \E[f(\tau_{\theta}^r)\, |\, \hat{\mathbb{X}}] = \int_0^T f(t)\, d\tilde{F}(t) + f(\infty)(1 - \tilde{F}(T))
 \end{align*}
 and therefore
 \begin{align*}
  \E[{Y}_{\tau_{\theta}^r \wedge S} \, |\, \hat{\mathbb{X}}] = \int_0^T {Y}_{t \wedge S}\, d\tilde{F}(t) + {Y}_S (1 - \tilde{F}(T)) = \int_0^S {Y}_{t}\, d\tilde{F}(t) + {Y}_S (1 - \tilde{F}(S)).
 \end{align*}
\end{proof}

\begin{remark}\label{rem:non-randomized-representation}
 In the deterministic case $Z = z > 0$ almost surely, we have
 \begin{align*}
  \tilde{F}(t) = \mathbbm{1}_{[z,\infty)} \left(\int_0^t \theta (\hat{\mathbb{X}}|_{[0,s]})^2 \, ds \right),
 \end{align*}
 thus
 \begin{align*}
  \E[{Y}_{\tau_{\theta}^r \wedge T}] = \E \left[ \int_0^T \mathbbm{1}_{[0,z)} \left(\int_0^t \theta (\hat{\mathbb{X}}|_{[0,s]})^2 \, ds \right) \, d{Y}_t \right] + \E[{Y}_0].
 \end{align*}
\end{remark}

\section{Linear signature stopping policies}\label{sec:lin_stopping_rues}

Using the regularization by randomization we will prove that it is enough to use stopping policies that are linear functionals of the signature.

\begin{definition} The space of \emph{linear signature stopping policies} $\mathcal{T}_{\mathrm{sig}} \subset \mathcal{T}$ is defined as
 \begin{align*}
  \mathcal{T}_{\mathrm{sig}} = \left\{ \theta \in \mathcal{T}\, :\, \exists l \in T((\R^{1+d})^*) \text{ such that } \theta(\hat{\mathbb{X}}|_{[0,t]}) = \langle l,\hat{\mathbb{X}}^{< \infty}_{0,t} \rangle \ \forall \hat{\mathbb{X}}|_{[0,t]} \in \Lambda_T \right\}.
 \end{align*}
Note that every $l \in T((\R^{1+d})^*)$ defines a $\theta_l \in \mathcal{T}$ by setting $\theta_l(\hat{\mathbb{X}}|_{[0,t]}) \coloneqq \langle l,\hat{\mathbb{X}}^{< \infty}_{0,t} \rangle$.
Let $Z$ be as in Definition~\ref{def:rand-stopping-times}, then we introduce the following notation for randomized stopping times associated to linear signature stopping policies 
    \begin{align}\label{eqn:randomized_sig_stopping_time}
      \tau^r_{l} \coloneqq \tau^r_{\theta_l} = \inf \left\{t \geq 0\,:\, \int_0^{t \wedge T} \langle l, \hat{\mathbb{X}}^{<\infty}_{0,s}\rangle^2 \, ds \geq Z \right\}.
    \end{align}
\end{definition}

As a consequence of the Stone-Weierstrass theorem, $\mathcal{T}_{\mathrm{sig}}$ is dense in $\mathcal{T}$. More precisely, we have
\begin{lemma}\label{lemma:sig_dense}
  Let $\P$ be a probability measure on $(\hat{\Omega}_T^p, \mathcal{B}(\hat{\Omega}_T^p))$. Then, for every $\varepsilon > 0$, there is a compact set $\mathcal{K} \subset \hat{\Omega}_T^p$ such that
  \begin{enumerate}
   \item $\P(\mathcal{K}) > 1 - \varepsilon$,
   \item $\mathcal{T}_{\mathrm{sig}}$, restricted of $\mathcal{K}$, is dense in $\mathcal{T}$. More precisely, for every $\theta \in \mathcal{T}$ there is a sequence $\theta_n \in \mathcal{T}_{\mathrm{sig}}$ such that
    \begin{align*}
  \sup_{\hat{\mathbb{X}} \in \mathcal{K};\ t \in [0,T]} |\theta_n(\hat{\mathbb{X}}|_{[0,t]}) - \theta(\hat{\mathbb{X}}|_{[0,t]}) | \to 0
 \end{align*}
 as $n \to \infty$.
  \end{enumerate}
 \end{lemma}
 \begin{proof}
    \cite[Lemma B.3]{KLP20}.
 \end{proof}

The main result of this section is the following
\begin{proposition}\label{prop:opt_sig_pol}
 Assume that $Z$ has a continuous density $\varrho$ and that $\E[ \lVert Y \rVert_{\infty}] < \infty$. Then
 \begin{align*}
  \sup_{\theta \in \mathcal{T}} \E[{Y}_{\tau_{\theta}^r \wedge T}] = \sup_{\theta \in \mathcal{T}_{\mathrm{sig}}} \E[{Y}_{\tau_{\theta}^r \wedge T}].
 \end{align*}
 It follows that
 \begin{align*}
  \sup_{\theta \in \mathcal{T}} \E[{Y}_{\tau_{\theta}^r \wedge T}] = \sup_{l \in T((\R^{1 + d})^*)} \E[{Y}_{\tau_l^r \wedge T}].
 \end{align*}

\end{proposition}

\begin{proof}
 It is enough to show that $\sup_{\theta \in \mathcal{T}} \E[{Y}_{\tau^r_{\theta} \wedge T}] \leq \sup_{\theta \in \mathcal{T}_{\mathrm{sig}}} \E[{Y}_{\tau^r_{\theta} \wedge T}]$. Let $\theta \in \mathcal{T}$. From Lemma \ref{lemma:sig_dense}, we know that for every $\varepsilon > 0$, there is a compact set $\mathcal{K} \subset \hat{\Omega}_T^p$ such that for $A \coloneqq \{\hat{\mathbb{X}} \in \mathcal{K} \}$, we have $\P(A) \geq 1 - \varepsilon$, and a sequence $\theta_n \in \mathcal{T}_{\mathrm{sig}}$ such that
 \begin{align}\label{eqn:prop_sigts}
  \lim_{n \to \infty} \sup_{\mathbb{X} \in \mathcal{K}; t \in [0,T]} |\theta_n(\mathbb{X}|_{[0,t]}) - \theta(\mathbb{X}|_{[0,t]})| = 0.
 \end{align}
 Let $\mathcal{K}$ be such a compact set, the precise choice will be made later. Set
 \begin{align*}
  \tilde{F}_n(t) = F_Z \left( \int_0^{t} \theta_n (\hat{\mathbb{X}}|_{[0,s]})^2 \, ds \right) \quad \text{and} \quad \tilde{F}(t) = F_Z \left( \int_0^{t}  \theta (\hat{\mathbb{X}}|_{[0,s]})^2 \, ds \right).
 \end{align*}
 Then,
 \begin{align*}
    &|\E[{Y}_T (1 - \tilde{F}_n(T))\, ;\, A] - \E[{Y}_T (1 - \tilde{F}(T))\, ;\, A]| \\
    \leq\ &\E[|{Y}_T| |\tilde{F}_n(T) - \tilde{F}(T)|\,;\, A] \\
    \leq\ &\E[|{Y}_T|] \sup_{\mathbb{X} \in \mathcal{K}} \left| F_Z \left( \int_0^{T} \theta_n (\mathbb{X}|_{[0,s]})^2 \, ds \right) - F_Z \left( \int_0^{T}  \theta (\mathbb{X}|_{[0,s]})^2 \, ds \right) \right|.
 \end{align*}
 Since $F_Z$ is continuous and uniformly continuous on compact sets,
  \begin{align*}
   \sup_{\mathbb{X} \in \mathcal{K}} \left| F_Z \left( \int_0^{T}  \theta_n (\mathbb{X}|_{[0,s]})^2 \, ds \right) - F_Z \left( \int_0^{T}  \theta (\mathbb{X}|_{[0,s]})^2 \, ds \right) \right| \to 0
  \end{align*}
  as $n \to \infty$. Indeed: we first show that
  \begin{align}\label{eqn:P_unif}
   \sup_{\mathbb{X} \in \mathcal{K}; t \in [0,T]} | \theta_n (\mathbb{X}|_{[0,t]})^2 - \theta (\mathbb{X}|_{[0,t]})^2| \to 0
  \end{align}
  as $n \to \infty$. Since $\sup_{n \geq 1} \sup_{\mathbb{X} \in \mathcal{K}; t \in [0,T]} |\theta_n(\mathbb{X}|_{[0,t]})| < \infty$, the functions $\theta$ and $\theta_n$ take their values in a compact set, hence \eqref{eqn:P_unif} follows from \eqref{eqn:prop_sigts}. Property \eqref{eqn:prop_sigts} also implies that
  \begin{align*}
   \sup_{\mathbb{X} \in \mathcal{K}} \left| \int_0^{T} \theta_n (\mathbb{X}|_{[0,s]})^2 \, ds - \int_0^{T} \theta (\mathbb{X}|_{[0,s]})^2 \, ds \right| \to 0
  \end{align*}
  as $n \to \infty$. Using continuity of $F_Z$ and uniform continuity on compact sets implies the claim. It follows that
  \begin{align*}
   \lim_{n \to \infty} |\E[{Y}_T (1 - \tilde{F}_n(T))\, ;\, A] - \E[{Y}_T (1 - \tilde{F}(T))\, ;\, A]| = 0.
  \end{align*}
  Since $|\tilde{F}_n(T) - \tilde{F}(T)| \leq 2$,
  \begin{align*}
   |\E[{Y}_T (1 - \tilde{F}_n(T))\, ;\, A^c] - \E[{Y}_T (1 - \tilde{F}(T))\, ;\, A^c]| \leq 2 \E[|{Y}_T|\,;\, A^c]
  \end{align*}
  and this quantity can me made arbitrarily small by the choice of $\mathcal{K}$.

  With the same arguments, we can show that
  \begin{align*}
   &\left| \E \left[  \int_0^T {Y}_{t} \left( \theta_n (\hat{\mathbb{X}}|_{[0,t]})^2 \varrho \left( \int_0^{t}  \theta_n (\hat{\mathbb{X}}|_{[0,s]})^2 \, ds \right) -  \theta (\hat{\mathbb{X}}|_{[0,t]})^2 \varrho \left( \int_0^{t}  \theta (\hat{\mathbb{X}}|_{[0,s]})^2 \, ds \right) \right)\, dt \right] \right| \\
   &\quad \to 0
  \end{align*}
  as $n \to \infty$ which implies the claim.
\end{proof}

Finally, we note that we do not need randomization for the approximation by stopping times based on signature stopping policies to work.

\begin{definition}
For $l \in T((\R^{1 + d})^*)$ define the hitting time of the signature against the half-plane orthogonal to $l$ by
  \begin{align*}
   \tau_l \coloneqq \inf\left\{t \in [0,T]\, :\, \langle l, {\hat{\mathbb{X}}}_{0,t}^{< \infty} \rangle \geq 1 \right\} .
  \end{align*}
\end{definition}

\begin{proposition}\label{thm:signature-stopping}
 Given $\E[\| Y \|_{\infty}] < \infty$, we have
 \begin{align*}
  \sup_{l \in T((\R^{1 + d})^*) } \E[Y_{\tau_l \wedge T}] = \sup_{\tau \in \mathcal{S}} \E[Y_{\tau \wedge T}].
 \end{align*}
\end{proposition}

\begin{proof}
 Using Proposition \ref{prop:opt_cont_pol} and \ref{prop:opt_sig_pol}, it suffices to show that
 \begin{align*}
   \sup_{l \in T((\R^{1+d})^*) } \E[Y_{\tau_{l}^r \wedge T}] \leq \sup_{l \in T((\R^{1 + d})^*) } \E[Y_{\tau_l \wedge T}].
 \end{align*}
   Choose $l \in T((\R^{1 + d})^*)$. Then
  \begin{align*}
   \E[ {Y}_{\tau_{l}^r \wedge T}\, |\, \hat{\mathbb{X}}] = \int_0^{\infty} {Y}_{\tau_z \wedge T}\, \P_Z(dz)
  \end{align*}
  where
  \begin{align*}
   \tau_z &\coloneqq \inf \left\{t \geq 0\,:\, \int_0^{t \wedge T} \langle l, \hat{\mathbb{X}}_{0,s}^{<\infty} \rangle ^2\, ds \geq z \right\} = \inf \left\{t \in [0,T] \,:\, \langle (l \shuffle l)\blue{1}/z, \hat{\mathbb{X}}_{0,t}^{<\infty} \rangle \geq 1 \right\}
  \end{align*}
  which is a signature hitting time for every $z > 0$ in the sense of the above definition. Taking expectation, we obtain
  \begin{align*}
   \E[ {Y}_{\tau_{l}^r \wedge T}] = \int_0^{\infty} \E[{Y}_{\tau_z \wedge T}] \, \P_Z(dz) \leq \sup_{\ell \in T((\R^d)^*) } \E[Y_{\tau_\ell \wedge T}]
  \end{align*}
  as claimed.
\end{proof}

\begin{remark}
 In the case of $X$ being a standard Markov process in $\R^d$ and $Y_t = G(t,X_t)$ for a continuous function $G$, it is known that
 \begin{align*}
  \sup_{\tau \in \mathcal{S}} \E[Y_{\tau \wedge T}] = \sup_{\tau \in \mathfrak{D}} \E[Y_{\tau \wedge T}]
 \end{align*}
 where $\mathfrak{D}$ denotes the set of all hitting times of closed sets in $\R^{1 + d}$ of the process $t \mapsto (t,X_t)$ \cite[Corollary 3 on p. 129]{Shi08}. Our Theorem can be seen as an extension of this classical result to non-Markovian processes.
\end{remark}

\section{Linearization of the optimal stopping problem}
\label{sec:appr-stopp-probl}
In this section, we will study randomized signature stopping times associated to linear stopping policies in $\mathcal{T}_{\mathrm{sig}}$ and a exponentially distributed random variable $Z \sim \operatorname{Exp}(1)$. 
From Proposition~\ref{lem:regularized_value} we have
\begin{equation}\label{eq:expected-exponential}
  \sup_{\tau \in \mathcal{S}} \E[Y_{\tau \wedge T}] = \sup_{l \in T((\R^{1+d})^*) } \E[Y_{\tau_l^r \wedge T}] = \sup_{l \in T((\R^{1+d})^*)} \E \left[ \int_0^T \exp \left( - \int_0^t \langle l, \hat{\mathbb{X}}_{0,s}^{< \infty} \rangle^2 \, ds \right)\, dY_t \right] + \E[Y_0].
\end{equation}

Recall that for group-like elements $\mathbf{a} \in G(V)$ as defined in~\eqref{eq:group}, polynomials of linear functionals in $\mathbf{a}$ can be expressed in terms of shuffle products of the linear functionals themselves, see~\eqref{eqn:lin_sig}. Consider the main term on the right-hand side of~\eqref{eq:expected-exponential}:
\begin{itemize}
\item The innermost term $\langle l, \hat{\mathbb{X}}_{0,s}^{< \infty} \rangle^2$ is a polynomial of a linear functional of the signature. It can, therefore, be expressed as a linear functional of the signature, more precisely, $\langle l, \hat{\mathbb{X}}_{0,s}^{< \infty} \rangle^2 = \langle l \shuffle l, \hat{\mathbb{X}}_{0,s}^{< \infty} \rangle$.
\item Given an element of the signature, integrating against a component of the underlying path produces another element of the signature. Concretely,
  \begin{equation*}
    \int_0^t \langle l, \hat{\mathbb{X}}_{0,s}^{< \infty} \rangle^2 ds = \langle (l \shuffle l) \blue{1}, \hat{\mathbb{X}}_{0,t}^{< \infty} \rangle,
  \end{equation*}
  recalling that the time-component of our driving path $\hat{X}_t = (t, X_t)$ was associated with the letter $\blue{1}$, see Remark~\ref{rem:path-letters}.
\item Next we need to apply the exponential function to $\langle (l \shuffle l) \blue{1}, \hat{\mathbb{X}}_{0,t}^{< \infty} \rangle$. Unfortunately, the exponential function is not a polynomial, so we cannot directly apply the shuffle product. However, as we shall see below, there is a corresponding \emph{exponential shuffle}, which comes with certain restrictions. Nonetheless, we shall see that we will still obtain a linear functional of the signature for our purposes.
\item Finally, we integrate against $Y$ and take the expectation. \textbf{If} $Y$ can itself be represented as a linear functional of the signature, integrating another linear functional of the signature against $Y$ will result in yet another linear functional of the signature. In this case, we can finally interchange the expectation, and the right-hand side of~\eqref{eq:expected-exponential} can be represented  as a $\sup$ over a linear functional of the \emph{expected signature} $\E[\hat{\X}^{<\infty}_{0,T}]$ of $\hat{X}$.
\end{itemize}
In the remainder of this section, we will follow through with this program.

We start with a  definition of an exponential function based on the shuffle product.
\begin{definition}
For $l\in T(V^{\ast})$ with $l=a_{0}\varnothing+\widetilde{l}$ and
$\langle \widetilde{l},\mathbf{1}\rangle =0$ we define the \emph{exponential shuffle}
\begin{equation}\label{expshuffledef}
\sexp(l):=\exp(a_{0})\sexp(\widetilde{l}),\text{ \ \ \ where
\ \ }\sexp(\widetilde{l}):=\sum_{r=0}^{\infty}\frac{1}{r!}
\widetilde{l}^{\shuffle r}.
\end{equation}
\end{definition}
Since obviously $\pi_{\leq N}(\widetilde{l}^{\shuffle r})=0$ for $r>N,$ the infinite sum is well defined as an element in the
extended tensor algebra $T((V^{\ast})).$ One may easily check that
$\sexp(\widetilde{l}_{1}+\widetilde{l}_{2})=\sexp
(\widetilde{l}_{1})\sexp(\widetilde{l}_{2})$ for $\widetilde{l}
_{1},\widetilde{l}_{2}\in T(V^{\ast})$ such that $\langle \widetilde
{l}_{1},\mathbf{1}\rangle =\langle \widetilde{l}_{2},\mathbf{1}\rangle =0.$
Thus, in particular, one has   \
\begin{align}\label{eqn:prod_rule_shuffle}
 \sexp(l_{1}+l_{2})=\sexp(l_{1})\sexp(l_{2})\text{
\ \ for all \  }l_{1},l_{2}\in T(V^{\ast}).
\end{align}

Note also that
  \begin{align*}
   \sexp(l) = \sum_{r=0}^{\infty}\frac{1}{r!}
l^{\shuffle r}
  \end{align*}
 for every $l\in T(V^{\ast})$.
 
We can now prove that the exponential shuffle linearizes the exponential
function for group-like elements. In this context, keep in mind that
$\sexp(l) \notin T(V^\ast)$ and, hence, it is not a well-defined linear
functional on $T((V))$. It is, however, trivially well-defined as a linear functional on $T(V)$, and, hence, can be applied to any projection $\pi_{\le N} (\mathbf{v})$, $\mathbf{v} \in T((V))$.
In addition, as the lemma shows, we can apply $\sexp(l)$ to
group-like elements.

\begin{lemma}\label{lem:key_shuffle_exp}
Let $l\in T(V^{\ast})$ and $\mathbf{g}\in G(V).$ One then has
\[
\left\vert \exp(\langle l,\mathbf{g}\rangle)-\left\langle \sexp
(l),\pi_{\leq N}(\mathbf{g})\right\rangle \right\vert \leq4\exp(\left\langle
l,1\right\rangle )\frac{\left(  \vert l\vert \vert\pi_{\leq\deg
(l)}(\mathbf{g})\vert\right)  ^{\left\lfloor N/\deg(l)\right\rfloor +1}
}{(\left\lfloor N/\deg(l)\right\rfloor +1)!}
\]
for  $N>2\deg(l)\vert {l}\vert \vert\pi_{\leq\deg(l)}
(\mathbf{g})\vert.$
\end{lemma}
\begin{proof}
Let us write $l=a_{0}\varnothing+\widetilde{l}$ with $\langle
\widetilde{l},\mathbf{1}\rangle =0,$ where for mutually different words
$w_{1},...,w_{n},$
\[
\widetilde{l}=\lambda_{1}w_{1}+...+\lambda_{n}w_{n},\text{  and set \ }
M:=\deg(l)=\max_{1\leq i\leq n}\deg(w_{i}),\text{ \ \ }m:=\min_{1\leq i\leq
n}\deg(w_{i})\geq1.
\]
We then have
\[
\pi_{\leq N}\left(  \sexp(\widetilde{l})\right)  =\sum_{r=0}^{\left\lfloor
N/M\right\rfloor }\frac{\widetilde{l}^{\shuffle r}}{r!}+\pi_{N}\left(
\sum_{r=\left\lfloor N/\deg(l)\right\rfloor +1}^{\left\lfloor N/m\right\rfloor
}\frac{\widetilde{l}^{\shuffle r}}{r!}\right)  \in T^{N}(V^{\ast}
)\text{\ \ \ for any \ \ }N\geq 1.
\]
Hence,
\begin{align*}
\left\langle \pi_{\leq N}\left(  \sexp(l)\right)  ,\mathbf{g}\right\rangle
& =\exp(a_{0})\sum_{r=0}^{\left\lfloor N/M\right\rfloor }\frac{\langle
\widetilde{l},\mathbf{g}\rangle^{r}}{r!}
 +\exp(a_{0})\sum_{r=\left\lfloor N/M\right\rfloor +1}^{\left\lfloor
N/m\right\rfloor }\frac{1}{r!}\langle\pi_{N}\left(  \widetilde{l}^{\shuffle
r}\right)  ,\mathbf{g}\rangle\\
& =:\exp(a_{0})\sum_{r=0}^{\left\lfloor N/M\right\rfloor }\frac{\langle
\widetilde{l},\mathbf{g}\rangle^{r}}{r!}+R_{N}^{(1)},%
\end{align*}
and since $\mathbf{g}\in G(V)$ it holds that
\begin{align*}
\langle\pi_{\leq N}\left(  \widetilde{l}^{\shuffle r}\right)  ,\mathbf{g}\rangle &
=\sum_{\substack{i_{1},...,i_{r}=1\\\deg(w_{i_{1}})+...+\deg(w_{i_{r}})\leq
N}}^{n}\lambda_{i_{1}}...\lambda_{i_{r}}\left\langle w_{i_{1}}\shuffle...\shuffle
w_{i_{r}},\mathbf{g}\right\rangle \\
& =\sum_{\substack{i_{1},...,i_{r}=1\\\deg(w_{i_{1}})+...+\deg(w_{i_{r}})\leq
N}}^{n}\left\langle \lambda_{i_{1}}w_{i_{1}},\mathbf{g}\right\rangle
...\left\langle \lambda_{i_{r}}w_{i_{r}},\mathbf{g}\right\rangle .
\end{align*}
Hence we have
\begin{align*}
\left\vert \langle\pi_{\leq N}\left(  \widetilde{l}^{\shuffle r}\right)
,\mathbf{g}\rangle\right\vert  & \leq\sum_{i_{1},...,i_{r}=1}^{n}\left\vert
\left\langle \lambda_{i_{1}}w_{i_{1}},\mathbf{g}\right\rangle \right\vert
...\left\vert \left\langle \lambda_{i_{r}}w_{i_{r}},\mathbf{g}\right\rangle
\right\vert =\left(  \sum_{i=1}^{n}\left\vert \left\langle \lambda_{i}%
w_{i},\mathbf{g}\right\rangle \right\vert \right)  ^{r}\\
& \leq\vert \widetilde{l}\vert ^{r}\left\vert \pi_{\leq\deg
(l)}(\mathbf{g})\right\vert ^{r}, \text{\ \ \ and so}  %
\end{align*}
\[
\left\vert R_{N}^{(1)}\right\vert \leq\exp(a_{0})\sum_{r=\left\lfloor
N/M\right\rfloor +1}^{\infty}\frac{\vert \widetilde{l}\vert
^{r}\left\vert \pi_{\leq\deg(l)}(\mathbf{g})\right\vert ^{r}}{r!}\leq2\exp
(a_{0})\frac{\left(  \vert \widetilde{l}\vert \vert\pi_{\leq\deg
(l)}(\mathbf{g})\vert\right)  ^{\left\lfloor N/M\right\rfloor +1}}{(\left\lfloor
N/M\right\rfloor +1)!}%
\]
for $N>N_{l,\mathbf{g}}:=2M\vert \vert\widetilde{l}\vert \pi_{\leq\deg
(l)}(\mathbf{g})\vert.$ One further has (note that $\mathbf{g}_{0}=\mathbf{1}$), due to a
similar estimation,
\begin{align*}
\exp(\langle l,\mathbf{g}\rangle)  & =\exp(a_{0})\exp(\langle\widetilde
{l},\mathbf{g}\rangle)=\exp(a_{0})\sum_{r=0}^{\left\lfloor N/M\right\rfloor
}\frac{\langle\widetilde{l},\mathbf{g}\rangle^{r}}{r!}+R_{N}^{(2)}\text{
\ \ with \ }\\
\left\vert R_{N}^{(2)}\right\vert  & \leq2\exp(a_{0})\frac{\left\vert
\langle\widetilde{l},\mathbf{g}\rangle\right\vert ^{\left\lfloor
N/M\right\rfloor +1}}{(\left\lfloor N/M\right\rfloor +1)!}\leq2\exp
(a_{0})\frac{\left(  \left\vert \widetilde{l}\right\vert \vert\pi_{\leq\deg
(l)}(\mathbf{g})\vert\right)  ^{\left\lfloor N/M\right\rfloor +1}}{(\left\lfloor
N/M\right\rfloor +1)!}%
\end{align*}
for $N>N_{l,\mathbf{g}}.$ Finally, by noting that $\left\langle \pi_{\leq N}\left(
\sexp(l)\right)  ,\mathbf{g}\right\rangle =\left\langle \exp^{\shuffle
}(l),\pi_{\leq N}\left(  \mathbf{g}\right)  \right\rangle ,$ and then taking all
together we obtain the stated result.
\end{proof}

\begin{remark}
  \label{rem:exp-shuffle-linearity}
  The equation $\langle \sexp(l), \mathbf{g} \rangle = \exp( \langle l, \mathbf{g}\rangle)$ is confusing at first glance, because $\mathbf{g} \mapsto \langle\sexp(l), \mathbf{g} \rangle$ \emph{seems} linear, whereas $\mathbf{g} \mapsto \exp( \langle l, \mathbf{g}\rangle)$ clearly is not. Note, however, that $\sexp(l) \in T((V^\ast))$ and, hence, does not define a linear map on $T((V))$. Indeed, the group $G(V)$ is not closed under linear combination, and, hence, Lemma~\ref{lem:key_shuffle_exp} does simply not apply to a linear combination of elements $\mathbf{g}_1,\mathbf{g}_2 \in G(V)$.
\end{remark}

The exponential shuffle satisfies a differential equation, which we shall use
later. Note that terms of the form
$\langle \sexp(l\blue{1}), \hat{\mathbb{X}}^{\leq N}_{0,t} \rangle$ are
(classically) differentiable in $t$.

\begin{lemma}\label{lem:shuffle_deriv}
  For every polynomial $l = \lambda_1 w_1 + \ldots + \lambda_n w_n \in T((\R^{1+d})^*)$,
  \begin{align*}
    \frac{d}{d t} \langle \sexp(l {\color{blue} 1}), \hat{\mathbb{X}}^{\leq N}_{0,t} \rangle = \sum_{i = 1}^n \langle \lambda_i w_i,  \hat{\mathbb{X}}^{< \infty}_{0,t} \rangle \langle \sexp(l {\color{blue} 1}), \hat{\mathbb{X}}^{\leq N - \operatorname{deg}(w_i) - 1}_{0,t} \rangle.
  \end{align*}
\end{lemma}

 \begin{proof}
   Note that
   \begin{align*}
     \frac{d}{dt} \langle w \blue{1}, \hat{\mathbb{X}}_{0,t}^{<\infty} \rangle = \frac{d}{dt} \int_0^t \langle w, \hat{\mathbb{X}}_{0,s}^{<\infty} \rangle \, ds = \langle w, \hat{\mathbb{X}}_{0,t}^{<\infty} \rangle
   \end{align*}
   for every word $w$. Hence, for $l = \lambda_1 w_1 + \ldots + \lambda_n w_n$, one always has $\langle l\blue{1},\mathbf{1}\rangle=0$ and so by (\ref{expshuffledef}),
   \begin{align*}
     \frac{d}{d t} &\langle \sexp(l {\color{blue} 1}), \hat{\mathbb{X}}^{\leq N}_{0,t} \rangle \\
     = \frac{d}{d t} &\sum_{0 \leq k_1 \operatorname{deg}(w_1 \blue{1}) + \ldots + k_n \operatorname{deg}(w_n\blue{1}) \leq N} \frac{\langle \lambda_1 w_1 {\color{blue} 1}, \hat{\mathbb{X}}^{< \infty}_{0,t}\rangle^{k_1}}{k_1!} \cdots \frac{\langle \lambda_n w_n {\color{blue} 1}, \hat{\mathbb{X}}^{< \infty}_{0,t}\rangle^{k_n}}{k_n!} \\
     = \phantom{\frac{d}{d t}} &\sum_{0 \leq k_1 \operatorname{deg}(w_1\blue{1}) + \ldots + k_n \operatorname{deg}(w_n\blue{1}) \leq N} \langle \lambda_1 w_1, \hat{\mathbb{X}}^{< \infty}_{0,t}\rangle \frac{\langle \lambda_1 w_1 \blue{1}, \hat{\mathbb{X}}^{< \infty}_{0,t}\rangle^{k_1 - 1}}{(k_1 - 1) !} \frac{\langle \lambda_2 w_2 {\color{blue} 1}, \hat{\mathbb{X}}^{< \infty}_{0,t}\rangle^{k_2}}{k_2!}\cdots \frac{\langle \lambda_n w_n {\color{blue} 1}, \hat{\mathbb{X}}^{< \infty}_{0,t}\rangle^{k_n}}{k_n!} \\
                   &\qquad  + \ldots + \langle \lambda_n w_n, \hat{\mathbb{X}}^{< \infty}_{0,t}\rangle \frac{\langle \lambda_1 w_1 {\color{blue} 1}, \hat{\mathbb{X}}^{< \infty}_{0,t}\rangle^{k_1}}{k_1!} \cdots \frac{\langle \lambda_{n-1} w_{n-1} {\color{blue} 1}, \hat{\mathbb{X}}^{< \infty}_{0,t}\rangle^{k_{n-1}}}{k_{n-1}!} \frac{\langle \lambda_n w_n \blue{1}, \hat{\mathbb{X}}^{< \infty}_{0,t}\rangle^{k_n-1}}{(k_n -1)!}
   \end{align*}
   and
   \begin{align*}
     &\sum_{0 \leq k_1\operatorname{deg}(w_1\blue{1}) + \ldots + k_n \operatorname{deg}(w_n\blue{1}) \leq N} \frac{\langle \lambda_1 w_1 \blue{1}, \hat{\mathbb{X}}^{< \infty}_{0,t}\rangle^{k_1 - 1}}{(k_1 - 1) !} \frac{\langle \lambda_2 w_2 {\color{blue} 1}, \hat{\mathbb{X}}^{< \infty}_{0,t}\rangle^{k_2}}{k_2!}\cdots \frac{\langle \lambda_n w_n {\color{blue} 1}, \hat{\mathbb{X}}^{< \infty}_{0,t}\rangle^{k_n}}{k_n!} \\
     =  &\sum_{0 \leq (k_1 + 1) \operatorname{deg}(w_1\blue{1}) + \ldots + k_n \operatorname{deg}(w_n\blue{1}) \leq N} \frac{\langle \lambda_1 w_1 \blue{1}, \hat{\mathbb{X}}^{< \infty}_{0,t}\rangle^{k_1}}{k_1 !} \frac{\langle \lambda_2 w_2 {\color{blue} 1}, \hat{\mathbb{X}}^{< \infty}_{0,t}\rangle^{k_2}}{k_2!}\cdots \frac{\langle \lambda_n w_n {\color{blue} 1}, \hat{\mathbb{X}}^{< \infty}_{0,t}\rangle^{k_n}}{k_n!} \\
     =\ &\langle \sexp(l {\color{blue} 1}), \hat{\mathbb{X}}^{\leq N - \operatorname{deg}(w_1\blue{1})}_{0,t} \rangle.\qedhere
   \end{align*}
 \end{proof}

We are now ready to formulate the main result of this section. Consider the optimization problem~\eqref{eq:expected-exponential}, which we modify by expressing the exponential by the exponential shuffle. Then we obtain convergence to the value of the optimal stopping problem. The proof requires us to localize w.r.t.~the rough path metric. Other than that, the below formulation is now essentially implementable: In particular, the result is formulated in terms of truncated signatures, which is necessary also from a numerical point of view.

 \begin{proposition}\label{thr:main_approx}
   For given $\kappa > 0$, we define the stopping time
   \begin{align*}
     S = S_{\kappa} = \inf \{ t  \geq 0\,:\, \| \hat{\mathbb{X}}\|_{p-var;[0,t]} \geq \kappa \} \wedge T.
   \end{align*}
   Assume $Z \sim \operatorname{Exp}(1)$ and
   $\E \left[ \| Y\|_{\infty} \right] < \infty$. Then
   \begin{align}\label{eqn:approx_res}
     \sup_{l \in T((\R^{1+d})^*)} \E[Y_{\tau_l^r \wedge T}] = \lim_{\kappa \to \infty} \lim_{K \to \infty} \lim_{N \to \infty} \sup_{|l| + \operatorname{deg}(l) \leq K} \E \left[\int_0^{S_\kappa} \langle \sexp( -(l \shuffle l){\color{blue} 1}), \hat{\mathbb{X}}^{\leq N}_{0,t} \rangle \, d Y_t \right] + \E[Y_0]
   \end{align}
   where the first two limit signs may be interchanged.
\end{proposition}

\begin{proof}
 To ease notation, assume that $Y_0 = 0$. Since
 \begin{align*}
  |Y_{\tau_{l} \wedge T} - Y_{\tau_{l} \wedge S}| \leq \sup_{|t-s| \leq T - S} |Y_{t} - Y_{s}| \to 0
 \end{align*}
 for every $l$ as $\kappa \to \infty$ and
 \begin{align*}
   \lim_{K \to \infty} \sup_{|l| + \operatorname{deg}(l) \leq K} \E[Y_{\tau_l \wedge \hat{S}}] = \sup_{l \in T((\R^{1+d})^*)} \E[Y_{\tau_l \wedge \hat{S}}],
 \end{align*}
 with $\hat{S}$ being either $S$ or $T$, it follows that
 \begin{align}\label{eqn:easy_approx}
   \sup_{l \in T((\R^{1+d})^*)} \E[Y_{\tau_l \wedge T}] = \lim_{\kappa \to \infty} \lim_{K \to \infty}  \sup_{|l| + \operatorname{deg}(l) \leq K} \E[Y_{\tau_l \wedge S}] =  \lim_{K \to \infty} \lim_{\kappa \to \infty}  \sup_{|l| + \operatorname{deg}(l) \leq K} \E[Y_{\tau_l \wedge S}].
 \end{align}
 Now fix $\kappa$, $K$ and $l$ with $|l| + \operatorname{deg}(l) \leq K$. Recall the estimate
 \begin{align*}
  \left| \int_0^T f(s)\, dg(s) \right| \leq T \|f'\|_{\infty} \|g\|_{\infty} + |f(T)g(T) - f(0)g(0)|.
 \end{align*}
Note that
 \begin{align*}
  \exp \left(- \int_0^{t} \langle l, \hat{\mathbb{X}}_{0,s}^{< \infty} \rangle^2 \, ds \right) = \exp( - \langle (l \shuffle l){\color{blue} 1}, \hat{\mathbb{X}}_{0,t}^{< \infty} \rangle ).
 \end{align*}
 Fix $N$. Then
 \begin{align*}
  &\left| \E \left[ \int_0^S \exp( - \langle (l \shuffle l){\color{blue} 1}, \hat{\mathbb{X}}_{0,t}^{< \infty} \rangle ) \, dY_t \right] - \E \left[ \int_0^{S} \langle \sexp( -(l \shuffle l){\color{blue} 1}), \hat{\mathbb{X}}^{\leq N}_{0,t} \rangle \, d Y_t \right] \right| \\
  \leq\ &(1 + T) \E \left[ \|Y\|_{\infty} \| \exp( - \langle (l \shuffle l){\color{blue} 1}, \hat{\mathbb{X}}_{0,\cdot}^{< \infty} \rangle ) -  \langle \sexp( -(l \shuffle l){\color{blue} 1}), \hat{\mathbb{X}}^{\leq N}_{0,\cdot} \rangle \|_{\mathcal{C}^1[0,S]} \right].
 \end{align*}
  Using Lemma \ref{lem:key_shuffle_exp},
  \begin{align*}
   \| \exp( - \langle (l \shuffle l){\color{blue} 1}, \hat{\mathbb{X}}_{0,\cdot}^{< \infty} \rangle ) -  \langle \sexp( -(l \shuffle l){\color{blue} 1}), \hat{\mathbb{X}}^{\leq N}_{0,\cdot} \rangle \|_{\infty;[0,S]} \leq 4 \sup_{t \in [0,S]} \frac{(|(l \shuffle l){\color{blue} 1}|  |\hat{\X}_{0,t}^{\leq 2\operatorname{deg}(l) + 1}|)^{M}}{M!}.
  \end{align*}
  where $M = \lfloor N/(2\operatorname{deg}(l) + 1) \rfloor + 1$ provided $N$ is sufficiently large. Clearly, $|(l \shuffle l){\color{blue} 1}| \leq C_K$. Using Lyons' Extension theorem \cite[Theorem 3.7]{LCL07}, we can estimate
  \begin{align*}
    \sup_{t \in [0,S]}|\hat{\X}_{0,t}^{\leq 2\operatorname{deg}(l) + 1}| \leq \|\hat{\X}^{\leq 2\operatorname{deg}(l) + 1}\|_{p-\mathrm{var};[0,S]} \leq C(1 + \|\hat{\X}\|_{p-\mathrm{var};[0,S]})^{2K + 1} = C(1 + \kappa)^{2K + 1}.
  \end{align*}
  Therefore, we obtain an estimate of the form
  \begin{align*}
   \| \exp( - \langle (l \shuffle l){\color{blue} 1}, \hat{\mathbb{X}}_{0,\cdot}^{< \infty} \rangle ) -  \langle \sexp( -(l \shuffle l){\color{blue} 1}), \hat{\mathbb{X}}^{\leq N}_{0,\cdot} \rangle \|_{\infty;[0,S]} \leq \frac{C^M}{M!}
  \end{align*}
  for a deterministic constant $C$.
  
  Next, we consider the derivatives. Set $\tilde{l} = -(l \shuffle l)$ and assume $\tilde{l} = \lambda_1 w_1 + \ldots + \lambda_k w_k$. Clearly,
  \begin{align*}
    \frac{d}{dt} \exp( \langle \tilde{l}{\color{blue} 1}, \hat{\mathbb{X}}_{0,t}^{< \infty} \rangle ) = \langle \tilde{l}, \hat{\mathbb{X}}_{0,t}^{< \infty} \rangle \exp(  \langle \tilde{l} {\color{blue} 1}, \hat{\mathbb{X}}_{0,t}^{< \infty} \rangle )
  \end{align*}
  and Lemma \ref{lem:shuffle_deriv} shows that
  \begin{align*}
   \frac{d}{dt} \langle \sexp( \tilde{l}{\color{blue} 1}), \hat{\mathbb{X}}^{\leq N}_{0,t} \rangle = \sum_{i = 1}^k \langle \lambda_i w_i,  \hat{\mathbb{X}}^{< \infty}_{0,t} \rangle \langle \sexp(\tilde{l} {\color{blue} 1}), \hat{\mathbb{X}}^{\leq N - \operatorname{deg}(w_i) -1}_{0,t} \rangle.
  \end{align*}
  Thus for $t \in [0,S]$,
  \begin{align*}
   &\left|\frac{d}{dt} \left(\exp( \langle \tilde{l}{\color{blue} 1}, \hat{\mathbb{X}}_{0,t}^{< \infty} \rangle ) - \langle \sexp( \tilde{l}{\color{blue} 1}), \hat{\mathbb{X}}^{\leq N}_{0,t} \rangle \right) \right| \\
   \leq\ &\sum_{i = 1}^k | \langle \lambda_i w_i,  \hat{\mathbb{X}}^{< \infty}_{0,t} \rangle| \left| \exp(  \langle \tilde{l} {\color{blue} 1}, \hat{\mathbb{X}}_{0,t}^{< \infty} \rangle ) - \langle \sexp(\tilde{l} {\color{blue} 1}), \hat{\mathbb{X}}^{\leq N - \operatorname{deg}(w_i) - 1}_{0,t} \rangle \right|.
  \end{align*}
  Using Lyons' Extension theorem,
  \begin{align*}
   | \langle \lambda_i w_i,  \hat{\mathbb{X}}^{< \infty}_{0,t} \rangle| \leq C |\lambda_i| \| \hat{\mathbb{X}} \|^{\operatorname{deg}(w_i)}_{p-\mathrm{var};[0,S]} \leq C |\lambda_i|
  \end{align*}
  for a deterministic constant $C>0$. Lemma \ref{lem:key_shuffle_exp} implies that for $N$ sufficiently large,
  \begin{align*}
   \left| \exp(  \langle \tilde{l} {\color{blue} 1}, \hat{\mathbb{X}}_{0,t}^{< \infty} \rangle ) - \langle \sexp(\tilde{l} {\color{blue} 1}), \hat{\mathbb{X}}^{\leq N - \operatorname{deg}(w_i) - 1}_{0,t} \rangle \right| \leq \frac{C^M}{M!}
  \end{align*}
  for a deterministic constant $C > 0$ and $M \to \infty$ as $N \to \infty$. It follows that also
  \begin{align*}
   \left\| \frac{d}{dt} \left( \exp( - \langle (l \shuffle l){\color{blue} 1}, \hat{\mathbb{X}}_{0,t}^{< \infty} \rangle ) -  \langle \sexp( -(l \shuffle l){\color{blue} 1}), \hat{\mathbb{X}}^{\leq N}_{0,t} \rangle \right) \right\|_{\infty;[0,S]} \leq \frac{C^{M}}{M!}.
  \end{align*}
  This implies that
  \begin{align*}
   &\sup_{|l| + \mathrm{deg}(l) \leq K} \left| \E \left[ \int_0^S \exp( - \langle (l \shuffle l){\color{blue} 1}, \hat{\mathbb{X}}_{0,t}^{< \infty} \rangle ) \, dY_t \right] - \E \left[ \int_0^{S} \langle \sexp( -(l \shuffle l){\color{blue} 1}), \hat{\mathbb{X}}^{\leq N}_{0,t} \rangle \, d Y_t \right] \right| \to 0
  \end{align*}
  as $N \to \infty$ and, in particular,
  \begin{align*}
    \lim_{N \to \infty} &\sup_{|l| + \mathrm{deg}(l) \leq K} \E \left[ \int_0^{S} \langle \sexp( -(l \shuffle l){\color{blue} 1}), \hat{\mathbb{X}}^{\leq N}_{0,t} \rangle \, d Y_t \right] \\
    = &\sup_{|l| + \mathrm{deg}(l) \leq K} \E \left[ \int_0^S \exp( - \langle (l \shuffle l){\color{blue} 1}, \hat{\mathbb{X}}_{0,t}^{< \infty} \rangle ) \, dY_t \right] = \sup_{|l| + \mathrm{deg}(l) \leq K} \E [Y_{\tau_l^r \wedge S} ].
  \end{align*}
  Together with \eqref{eqn:easy_approx}, this proves \eqref{eqn:approx_res}.
\end{proof}

Often, one is interested to solve the stopping problem for specific functionals of the underlying process $X$. In the next corollary, we consider a particular example. To simplify the exposition, we will consider the case $d=1$ only. The generalization to arbitrary dimensions $d$ is straightforward.

  \begin{corollary}\label{cor:opt-stopping-expected-sig}
  Assume $d = 1$ and that
  \begin{align*}
   Y_t = G(X_t) + \int_0^t L(X_s)\, ds
  \end{align*}
  for polynomials $G$ and $L$. Then
   \begin{align*}
   &\sup_{l \in T((\R^{1+d})^*)} \E[Y_{\tau_l^r \wedge T}] \\
   &= \lim_{\kappa \to \infty} \lim_{K \to \infty} \lim_{N \to \infty}  \sup_{|l| + \operatorname{deg}(l) \leq K}\langle (\sexp( -(l \shuffle l){\color{blue} 1}) \shuffle  G^{' \shuffle}({\color{blue} 2})) {\color{blue} 2} + (\sexp( -(l \shuffle l){\color{blue} 1}) \shuffle  L^{\shuffle}({\color{blue} 2})) {\color{blue} 1} , \E [\hat{\mathbb{X}}^{\le N}_{0,S} ] \rangle \\
   &\qquad + \E[Y_0].
  \end{align*}
  In particular, if $d = 1$ and $X_0 = 0$,
   \begin{align}\label{eqn:approx_stopping}
        \sup_{l \in T((\R^{2})^*)} \E[X_{\tau_l^r \wedge T}] = \lim_{\kappa \to \infty} \lim_{K \to \infty} \lim_{N \to \infty} \sup_{|l| + \operatorname{deg}(l) \leq K} \langle \sexp( -(l \shuffle l){\color{blue} 1}){\color{blue} 2}, \E [\hat{\mathbb{X}}^{\leq N}_{0,S} ] \rangle .
   \end{align}
  \end{corollary}

\begin{proof}
 We have
  \begin{align*}
   \int_0^{S} \langle \sexp( -(l \shuffle l){\color{blue} 1}), \hat{\mathbb{X}}^{\leq N}_{0,t} \rangle \, d Y_t &= \int_0^{S} \langle \sexp( -(l \shuffle l){\color{blue} 1}), \hat{\mathbb{X}}^{\leq N}_{0,t} \rangle G'(X_t) \, d X_t \\
   &\quad + \int_0^{S} \langle \sexp( -(l \shuffle l){\color{blue} 1}), \hat{\mathbb{X}}^{\leq N}_{0,t} \rangle L(X_t) \, dt.
  \end{align*}
  Since $G'$ is a polynomial,
  \begin{align*}
    \int_0^{S} \langle \sexp( -(l \shuffle l){\color{blue} 1}), \hat{\mathbb{X}}^{\leq N}_{0,t} \rangle G'(X_t) \, d X_t &= \int_0^{S} \langle \sexp( -(l \shuffle l){\color{blue} 1}), \hat{\mathbb{X}}^{\leq N}_{0,t} \rangle \langle G^{' \shuffle}({\color{blue} 2}), \hat{\mathbb{X}}^{< \infty}_{0,t} \rangle \, d X_t \\
    &=  \int_0^{S} \langle \pi_{\leq N} (\sexp( -(l \shuffle l){\color{blue} 1})) \shuffle  G^{' \shuffle}({\color{blue} 2}), \hat{\mathbb{X}}^{< \infty}_{0,t} \rangle \, d X_t \\
    &=  \int_0^{S} \langle \pi_{\leq N + \operatorname{deg}(G')} (\sexp( -(l \shuffle l){\color{blue} 1}) \shuffle  G^{' \shuffle}({\color{blue} 2})), \hat{\mathbb{X}}^{< \infty}_{0,t} \rangle \, d X_t \\
    &= \langle (\sexp( -(l \shuffle l){\color{blue} 1}) \shuffle  G^{' \shuffle}({\color{blue} 2})) {\color{blue} 2}, \hat{\mathbb{X}}^{\leq N + \operatorname{deg}(G') + 1}_{0,S} \rangle.
  \end{align*}
  Similarly, since $V$ is a polynomial,
  \begin{align*}
   \int_0^{S} \langle \sexp( -(l \shuffle l){\color{blue} 1}), \hat{\mathbb{X}}^{\leq N}_{0,t} \rangle L(X_t) \, dt = \langle (\sexp( -(l \shuffle l){\color{blue} 1}) \shuffle  L^{\shuffle}({\color{blue} 2})) {\color{blue} 1}, \hat{\mathbb{X}}^{\leq N + \operatorname{deg}(V) + 1}_{0,S} \rangle.
  \end{align*}
  Taking expectation, we obtain
  \begin{align*}
   \E \left[\int_0^{S} \langle \sexp( -(l \shuffle l){\color{blue} 1}), \hat{\mathbb{X}}^{\leq N}_{0,t} \rangle \, d Y_t \right] &= \langle (\sexp( -(l \shuffle l){\color{blue} 1}) \shuffle  G^{' \shuffle}({\color{blue} 2})) {\color{blue} 2}, \E[\hat{\mathbb{X}}^{\leq N + \operatorname{deg}(G') + 1}_{0,S}] \rangle \\
   &\quad + \langle (\sexp( -(l \shuffle l){\color{blue} 1}) \shuffle  L^{\shuffle}({\color{blue} 2})) {\color{blue} 1}, \E[\hat{\mathbb{X}}^{\leq N + \operatorname{deg}(V) + 1}_{0,S}] \rangle.
  \end{align*}
  Using Theorem \ref{thr:main_approx}, we can deduce the result.
\end{proof}

\begin{remark} Once a Monte Carlo simulation of the expected signature of the rough path is generated, the optimization problems corresponding to the suprema in Corollary~\ref{cor:opt-stopping-expected-sig} are \emph{purely deterministic}.
Indeed, the objective function that needs to be maximized is a polynomial in the coefficients of the word $l$.
In Section \ref{sec:linear_numerics} we discuss the  implementation details for this step.
\end{remark}
\begin{remark}
When the payoff process $Y$ is a martingale then it follows from Doob's optional sampling theorem that optimal stopping value is zero.
It is instructive to demonstrate how this fact can be observed in the form of the optional stopping problem in Corollary~\ref{cor:opt-stopping-expected-sig}.
When $Y$ is a continuous martingale then it also admits a lift to a geometric rough path and therefore we may assume that $X=Y$.
Now let $w$ be an arbitrary word, then it holds
\begin{align*}
\langle w\blue{12}, \hat\X^{<\infty}_{0, t}\rangle = \int_0^{t} \left(\int_0^{s}\langle w, \hat\X^{<\infty}_{0, u} \rangle du\right) \circ dX_s = \int_0^{t} \left(\int_0^{s}\langle w, \hat\X^{<\infty}_{0, u} \rangle du \right) dX_s,
\end{align*}
where we have used that the Stratonovich-Itō correction is zero due to the finite variation of the integrand.
By the martingale property of the integral in the above right hand side it follows that
 $\langle w\blue{12}, \E[\hat\X^{\le N}_{0, S}]\rangle = 0$ for all $N\ge 1$.
Therefore, we see that all terms in the supremum in the right hand side of \eqref{eqn:approx_stopping} vanish.
\end{remark}

\begin{remark}\label{rem:american-option-on-signature}
  Note that similar formulas are available whenever $Y$ is roughly given as a polynomial of the signature. We restrict ourselves to a representative class of examples below.
  We note here that payoffs of American options usually cannot be exactly represented in such a way. In particular, for the standard American (put) option, we have $Y_t = (K - X_t)_+$ for some $K > 0$, where $X$ denotes the underlying asset price process.
  If we want to price American options using signature stopping policies, we have two possible remedies. We can approximate the payoff function by polynomials, which would allow us to directly apply Corollary~\ref{cor:opt-stopping-expected-sig}.
  Alternatively, we can attach $Y$ to the path $X$, i.e., consider $\tilde{X}_t \equiv (t, X_t, Y_t)$. Then, the corollary applies trivially, but at the price of increasing the dimension of the state space. The same strategy also works for more complicated functionals of the rough path $\hat{\mathbb{X}}$. For instance, $Y$ can be of the form $Y_t = g(t,\tilde{Y}_t)$ where $\tilde{Y}$ solves a rough differential equation
  \begin{align*}
   d \tilde{Y}_t = b(\tilde{Y}_t)\, dt + \sigma(\tilde{Y}_t)\, d \mathbb{X}_t.
  \end{align*}
  If $g$ is sufficiently smooth, $Y$ is \emph{controlled by $\mathbb{X}$} (cf. \cite{FH14}) which guarantees that $\tilde{X}_t \equiv (t, X_t, Y_t)$ can be lifted to a rough paths valued process.
\end{remark}

\section{Deep signature stopping policies}\label{sec:non-linear-rules}
In the two preceding sections we focused on stopping policies $\theta$ that are linear functionals of the signature.
Proposition~\ref{prop:opt_sig_pol} shows that this class of policies is sufficient to approximate a optimal stopping policy.
Regarding numerical approximation and in order to obtain reasonable results it might however be necessary to step deeply into the signature, that is to approximate the signature upto a high truncation level.
In this section we propose a different model for the stopping policy, in which the signature serves as an input feature to a deep neural network.
Proposition~\ref{lem:regularized_value} from above is the key link for defining a smooth loss-function for these models, which will be discussed in more detail in Section \ref{sec:non-linear-loss}.
In this framework, the shuffle property of the signature is not directly needed and therefore it is reasonable to replace the signature by the log-signature, which allows for a low dimensional representation as explained below.

We will again assume throughout the section that $\X\in\Omega^{p}_T$ and that $Y = (Y_t)_{0 \le t \le T}$ is a real-valued continuous process adapted to the filtration generated by $\X$.
\begin{definition}[Log-signature]
We define the log-signature $\L^{<\infty}: [0,T] \to T((V))$ by
\begin{align*}
\L^{<\infty}_{t} = \tlog(\X^{<\infty}_t), \quad 0 \le t \le T,
\end{align*}
where $\log_\otimes$ is the tensor logarithm defined by
\begin{align*}
\tlog(\mathbf{1}+ \mathbf{a}) = \sum_{k =0}^\infty \frac{(-1)^{k+1}}{k}\mathbf{a}^{\otimes k} \in T((V)), \quad \mathbf{a} \in T((V)), \; \langle \varnothing, \mathbf{a} \rangle = 0.
\end{align*}
\end{definition}
We will also write $\hat\L^{<\infty} = \tlog(\hat\X^{<\infty})$ for the log-signature of the time-augmented path $\hat\X$.
Recall that the signature is an element of the group $G(V)$.
Therefore, the log-signature takes values in the set $$\mathfrak{g}(V) := \tlog(G(V)) = \left\{\tlog(\mathbf{g}) \;\vert\; \mathbf{g} \in G(V)\right\} \subset T((V)),$$ which forms a Lie-algebra with the bracket given by the commutator
\begin{align*}
[\mathbf{a}, \mathbf{b}] = \mathbf{a}\otimes\mathbf{b} - \mathbf{b}\otimes\mathbf{a}.
\end{align*}
In fact, $\mathfrak{g}(V)$ is the free Lie-algebra over $V$ (see e.g. \cite[Section 7.3]{FV10}) and therefore the log-signature can be expressed in terms of iterated Lie brackets of the basis elements $\{ e_1, \cdots, e_d\}$, i.e. there exists explicit coefficients $(\lambda_{i_1, \dots, i_n}(t))$ such that
\begin{align*}
\L^{<\infty}_{0, t} = \sum_{n =1}^{\infty} \sum_{i_1, \dots, i_n =1}^{d} \lambda_{i_1, \dots, i_n}(t)[e_{i_1}, [ e_{i_2} , [ \dots, [e_{i_{n-1}}, e_{i_n}] \dots ]].
\end{align*}
This result already goes back to \cite{chen1957integration}.
However, the above representation is not very efficient, due to linear dependencies (e.g. $[e_1, e_2] = - [e_2, e_1]$).
An efficient basis is given by the \emph{Lyndon basis} or more generally by a \emph{Hall basis}. 
We refer to \cite{R93} for a construction of such a basis.
Define $$\mathfrak{g}^N(V) = \pi_{\le N}(\mathfrak{g}(V)) = \tlog(G^{N}(V)),$$ the free step-$N$ nilpotent Lie-algebra, then we have the following
\begin{proposition}
For $N \in \{1, 2, \dots\}$ the dimension of $\mathfrak{g}^N(\R^{d})$ as a linear vector space is given by
\begin{align}\label{eq:def_eta}
\eta_{d, N} \coloneqq \sum_{n=1}^{N} \frac{1}{n}\sum_{k|n}\mu(n/k)d^{k},
\end{align}
where the inner sum is taken over all divisors $k$ of $n$ and $\mu$ is the Möbius function.
\end{proposition}
\begin{proof}
The statement directly follows from the existence of a Hall basis of the free Lie-algebra and its dimensionality, which can be found in \cite[Theorem 6]{R93}.
\end{proof}

\begin{table}
\renewcommand{\arraystretch}{1.1}
\begin{tabular}{c | c | c | c | c | c | c | c | c }
  & \multicolumn{2}{c|}{$d=1$} &  \multicolumn{2}{c|}{$d=2$} &  \multicolumn{2}{c|}{$d=3$} &  \multicolumn{2}{c}{$d=4$} \\
  & $\sigma_{d,N}$ & $\eta_{d,N}$ & $\sigma_{d,N}$ & $\eta_{d,N}$ & $\sigma_{d,N}$ & $\eta_{d,N}$ & $\sigma_{d,N}$ & $\eta_{d,N}$  \\\hline
$N = 1$ & 1 & 1 & 2		& 2		& 3		& 3		& 4		& 4		\\
$N = 2$ & 2 & 1 & 6		& 3		& 12		& 6		& 20		& 10		\\
$N = 3$ & 3 & 1 & 14		& 5		& 39		& 14		& 84		& 30		\\
$N = 4$ & 4 & 1 & 30		& 8		& 120	& 32		& 340	& 90		\\
$N = 5$ & 5 & 1 & 62		& 14		& 363	& 80		& 1364& 294	\\
$N = 6$ & 6 & 1 & 126	& 23		& 1092& 196	& 5460& 964	\\
\end{tabular}
\vspace{0.5em}
\caption{Comparisons of the dimensionality of the representations of the truncated signature $\sigma_{d,N}$ and truncated log-signature $\eta_{d,N}$. (Values taken from \cite{reizenstein2017calculation}.)} \label{tab:dimensions}
\end{table}

Note that dimension of the truncated tensor algebra $T^{N}(\R^{d})$ is given by $ (d^{N+1} - 1)/(d-1)$.
Since the truncated signature always starts with $\mathbf{1}$, its representation in the linear space is effectively of dimension $\sigma_{d,N} = (d^{N+1} - d)/(d-1)$.
In Table \ref{tab:dimensions} we compare the values $\sigma_{d,N}$ and $\eta_{d,N}$ for different values of $d$ and $N$.
While both values have the same asymptotic growth when $N \to \infty$, the efficient representation of the log-signature allows for numerical tractability up to higher truncations levels $N$.
Finally note that the Baker-Campbell-Hausdorff formula, and therefore the log-signature of a piecewise linear path, can be directly calculated in a Hall basis. 
For an algorithm see for instance \cite{casas2008efficient} and the summarizing note \cite{reizenstein2017calculation}.

We are now ready to present our class of deep signature stopping policies.
\begin{definition}
Define $\mathcal{T}_{\log} \subset \mathcal{T}$ to be the set of continuous stopping policies $\theta: \Lambda_T \to \R$ of the form 
\begin{align*}
\theta(\hat\X\vert_{[0,t]}) = \left(\theta_{\log} \circ \tlog \right) (\hat\X_{0,t}^{\le N}),
\end{align*}
for some $N \in \{1, 2, \dots\}$ and where $\theta_{\log}: \mathfrak{g}^{N}(\R^{d+1}) \to \R$ is a deep neural network of the form
\begin{align}\label{eq:networks}
\theta_{\log} = A_0 \circ \varphi \circ A_1 \circ \varphi \circ \cdots \circ A_I,
\end{align}
where $A_I: \mathfrak{g}^{N}(\R^{d+1}) \to \R^{q}$, $A_i: \R^{q} \to \R^{q}$, $1 \le i < I$, $A_0: \R^{q}\to\R$, are affine maps for some $q, I \in \{1, 2, \dots\}$ and $\varphi$ is an activation function, i.e. continuous and not a polynomial (e.g. $\varphi(x)_i = \max\{x_i, 0\}$ for $1 \le i \le q$ and $x\in\R^{q}$).
\end{definition}

The following proposition, which is a consequence of  Proposition \ref{prop:opt_cont_pol}, Proposition \ref{prop:opt_sig_pol} and the universal approximation theorem for neural networks, states that the optimization over stopping policies in the class $\mathcal{T}_{\log}$ is sufficient.

\begin{proposition}\label{prop:log_signature-stopping}
 Given $\E[\| Y \|_{\infty}] < \infty$, we have
 \begin{align*}
  \sup_{ \theta \in \mathcal{T}_{\log} } \E[Y_{\tau_\theta^{r} \wedge T}] = \sup_{\tau \in \mathcal{S}} \E[Y_{\tau \wedge T}].
 \end{align*}
\end{proposition}
\begin{proof}
Since $\mathcal{T}_{\log} \subset \mathcal{T}$ it follows from Proposition \ref{prop:opt_cont_pol} and \ref{prop:opt_sig_pol} that it suffices to show that
 \begin{align*}
  \sup_{\theta \in \mathcal{T}_{\mathrm{sig}}} \E[Y_{\tau_\theta^{r} \wedge T}] \le \sup_{ \theta \in \mathcal{T}_{\log} } \E[Y_{\tau_\theta^{r} \wedge T}].
\end{align*}
Therefore, let $\theta \in \mathcal{T}_{\mathrm{sig}}$, i.e. there exists some $l \in T((\R^{d+1})^{\ast})$ such that $$\theta(\X\vert_{[0,t]}) = \langle l, \hat\X^{<\infty}_{0,t} \rangle = \langle l, \hat\X^{\le N}_{0,t} \rangle,$$
where $N := \max\{\deg(l), [p]\}$.
We equip $G^{N}(\R^{d+1}) \subset T^{N}(\R^{d+1})$ with the subspace topology.
Note that the inverse of the tensor logarithm $\tlog$ is the tensor exponential map $\texp$, which is defined analogously by its power series expansion (see also Example~\ref{exmpl:brownian-motion}).
Now define the continuous map 
$$\psi: \mathfrak{g}^{N}(\R^{d+1}) \to \R, \quad \mathbf{x} \mapsto \langle l, \texp(\mathbf{x})\rangle$$
Let $K \subset G^{N}(\R^{d+1})$ be an arbitrary compact set and define $K^{\prime} \coloneqq \tlog(K) \subset \mathfrak{g}^{N}(\R^{d+1})$, which is also compact.
From \cite[Theorem 1]{leshno1992multilayer} it follows that there exists as sequence of functions $(\psi_{n})_{n\ge1}$ of the form \eqref{eq:networks} that converge uniformly on $K^{\prime}$ towards $\psi$.
Hence, it also follows that the sequence of maps $(\psi_n\circ\tlog)_{n \ge 1}$ converge uniformly on $K$ towards the map $\langle l, \cdot \rangle: G^{N}(\R^{d+1}) \to \R$.

Next, note that the Lyons-lift map 
$$\Lambda_T \to G^{N}(\R^{d+1}), \quad \hat\X\vert_{[0,t]} \mapsto \hat\X
_{0,t}^{\le N}$$
is continuous (see \cite[Section 9.1]{FV10}) and therefore it follows from above that for any compact set $\mathcal{K} \subset \hat\Omega_T^{p}$, there exists a sequence $(\theta_n)_{n=1}^{\infty} \subset \mathcal{T}_{\log}$ such that
\begin{align*}
\lim_{n \to \infty}\sup_{\X\in\mathcal{K}, t \in [0,t]}| \theta_n(\X\vert_{[0,t]}) - \theta(\X\vert_{[0,t]})| = 0.
\end{align*}
The rest of the proof is now completely analogous to proof of Proposition \ref{prop:opt_sig_pol}.
\end{proof}

\section{Numerical Examples}\label{sec:numerics}
In this section we are going to evaluate the performance of the signature stopping methodology for the numerical approximation of the optimal stopping value.
We have particularly chosen examples in which the underlying rough path is non-Markov.
The technical steps that are necessary to implement a Monte-Carlo simulation are discussed in detail in Appendix \ref{sec:implmentation_details}.

\subsection{Optimal stopping of fractional Brownian motion}
Let $X^{H}$ be a fractional Brownian motion with Hurst parameter $H\in (0,1]$ and recall from Example~\ref{exmpl:fbm} 
that the time augmented process $\hat{X}^{H}$ has a lift to a geometric rough path.
Also recall that in the case $H = 0.5$ the process $X^{H}$ is a standard Brownian motion and otherwise it is not a semimartingale and not a Markov process.
We are considering the optimal stopping problem with the payoff given by the underlying process itself, i.e. $Y\equiv X^{H}$.
The payoff clearly satisfies the condition $\E(\Vert Y\Vert_{\infty; [0,T]})<\infty$ for any $T<\infty$ and therefore our approximation results apply.
We set $T=1$ and use a time discretization with $J=100$ steps for Monte-Carlo simulation of the payoff and the signature (see Section~\ref{sec:time_discretization} for more detail).

This example was recently studied in \cite{becker2019deep} and we will use the there presented values as benchmarks.
The methodology in \cite{becker2019deep} was developed for Markov processes and is based on a parametrization of the optimal stopping policy by deep neural networks, one for each time in the discretization grid.
The networks are then trained in a backwards induction using the dynamic programming principle.
The methodology was applied to the fractional Brownian motion example by lifting the time-discretized process to a $100$-dimensional Markov process including the entire past of the process.

\textbf{Theoretical observations.} Before evaluating the numerical results we present a few theoretical observations about the value of the optimal optimal stopping problem of the fractional Brownian motion.

In the case $H=0.5$, the payoff process $Y$ is a standard Brownian motion and hence a martingale. 
Therefore it follows from Doob's optional sampling theorem that the true value of the optimal stopping problem is zero.

The inclusion of the case $H=1$ in the definition of the fractional Brownian motion is a matter of convention, in any case, the right-hand side of \eqref{eq:cov_fbm}  defines the positive semidefinite covariance kernel $\rho_1(s,t) =  s\cdot t$ for $s,t \ge 0$.
The process $X := (t \cdot \xi)_{t \ge 0}$ with a standard normal distributed random variable $\xi$ admits this covariance structure. 
Since the sample paths of this process are given by straight lines, one can easily verify that an optimizing sequence of stopping times $(\tau_n)_{n\ge 1}$ is given by
\begin{align*}
\tau_n = \left\{
\begin{array}{cc}
T, & X_{1/n} > 0 \\
1/n, & X_{1/n} \le 0
\end{array}
\right.,\quad  n = 0, 1, 2, \dots
\end{align*}
The optimal stopping value for $H=1$ is therefore given by $\lim_{n\to \infty}\E[Y_{\tau_n}] = \E[T\cdot(\xi)_+] = T/ \sqrt{2\pi}$.
Furthermore, we also see that the value of the time discretized problem is given by $\E[Y_{\tau_{J}}] = 0.99 / \sqrt{2\pi} \approx 0.395$ for $J=100$ and $T=1$.
We point out that this discussion of the case $H=1$ was already presented in \cite{becker2019deep}.

On the other side of the spectrum, that is when $H$ approaches $0$, the covariance kernel of the fractional Brownian motion \eqref{eq:cov_fbm} converges pointwise to the following kernel
\begin{align}\label{eq:H0covariance}
\rho_0(s, t) = \left\{\begin{array}{cc}
0, & t= 0 \vee s = 0.\\
1, & t = s\wedge s > 0,\\
\frac{1}{2}, & \text{else.}
\end{array}
\right.,\quad  t,s \ge 0.
\end{align}
The Kolmogorov extension theorem implies that a Gaussian process with this covariance structure exist, however such a process is not measurable in the usual sense of stochastic processes (see \cite[Section 19.5]{stoyanov2014counterexamples}) and therefore the associated optimal stopping problem is not well defined.
See also \cite{fyodrov2016fractional}, \cite{neuman2018fractional}, \cite{hager2020multiplicative} and \cite{bayer2020log} for regularizations and normalizations of the fractional Brownian motion and the convergence as $H$ approaches $0$.
For the comparison with numerical approximations it is nevertheless interesting to study the discrete time optimal stopping problem of a process with the covariance structure $\rho_0$.
Let $(\xi_{0}, \xi_1, \dots, \xi_J)$ be a finite sequence of independent standard normal distributed random variables and note that the process $$\tilde{Y}_{j} :=(\xi_j - \xi_0)/\sqrt{2}, \quad j= 0,1,\dots, J,$$ has the covariance structure $$\E\big[\tilde{Y}_{i} \tilde{Y}_{j}\big] = \rho_0(i, j) =\rho_0(t_i, t_j), \quad i,j \in\{ 0, 1, \dots, J\}, \;\; t_j  = Tj/J.$$
In Appendix \ref{sec:H0calcuation}  we present an explicit solution to the discrete time optimal stopping problem associated to the process $\tilde Y$.
For $J=100$ the optimal stopping value is approximately given by $1.5830$ and for $J\to \infty$ the value converges to infinity (see Proposition~\ref{prop:H0} and Remark~\ref{rem:H0Jinfty}).
At this point, let us also mention that the value of the discrete time stopping problem depends continuously on the Hurst parameter $H\in[0,1]$.
Indeed, one has to verify that the distribution of the process $(X^{H}_{t_j})_{j = 0, 1, \dots, J}$ depends continuously on $H$ in the \emph{weak adapted topology} (see \cite{backhoff2020all}).
However, this is a consequence of the fact that the process is Gaussian and the densities of the law (receptively conditional laws) are everywhere continuous in $H$.
Since the value of the discrete time optimal stopping problem is a lower-bound to the value of the continuous time problem we have the following

\begin{proposition}\label{prop:Hto0} The value of the optimal stopping problem of the fractional Brownian motion $(X^{H}_t)_{0 \le t \le T}$ diverges as $H\to0$, more precisely we have
\begin{align*}
\lim_{H\to 0}\sup_{\tau \in \mathcal{S}}\E\left[X^{H}_{\tau \wedge T} \right] = \infty.
\end{align*}
\end{proposition}

\begin{table}
  \begin{tabular}{c || c | c | c || c|c|c|c|c || c | c}
 	& \multicolumn{3}{c||}{lin. signature policies} & \multicolumn{5}{c||}{deep signature policies}& \multicolumn{2}{c}{BCJ}\\
 $H$ &  $N = 3$ &  $N = 4$ &  $N = 5$ &  $N = 1$ &  $N = 2$ &  $N = 3$ &  $N = 4$ &  $N = 5$ & low     & up     \\\hline
 0.0 &    1.161 &    1.331 &    1.430 &    1.384 &    1.582 &    1.583 &    1.583 &    1.583 &    -    &     -  \\
 0.1 &    0.760 &    0.858 &    0.916 &    0.929 &    1.030 &    1.038 &    1.042 &    1.043 & $1.048$ & $1.049$\\
 0.2 &    0.468 &    0.523 &    0.553 &    0.584 &    0.637 &    0.646 &    0.649 &    0.651 & $0.658$ & $0.659$\\
 0.3 &    0.257 &    0.285 &    0.293 &    0.329 &    0.355 &    0.361 &    0.363 &    0.364 & $0.369$ & $0.380$\\
 0.4 &    0.107 &    0.116 &    0.118 &    0.140 &    0.149 &    0.152 &    0.152 &    0.153 & $0.155$ & $0.158$\\
 0.5 &    0.000 &    0.000 &    0.000 &    0.000 &    0.000 &    0.000 &    0.000 &    0.000 & $0.000$ & $0.005$\\
 0.6 &    0.081 &    0.084 &    0.085 &    0.106 &    0.112 &    0.113 &    0.114 &    0.114 & $0.115$ & $0.118$\\
 0.7 &    0.146 &    0.152 &    0.154 &    0.191 &    0.198 &    0.201 &    0.201 &    0.202 & $0.206$ & $0.207$\\
 0.8 &    0.199 &    0.207 &    0.208 &    0.263 &    0.269 &    0.272 &    0.273 &    0.273 & $0.276$ & $0.278$\\
 0.9 &    0.247 &    0.256 &    0.258 &    0.328 &    0.332 &    0.334 &    0.334 &    0.334 & $0.336$ & $0.339$\\
 1.0 &    0.296 &    0.307 &    0.309 &    0.395 &    0.395 &    0.395 &    0.395 &    0.395 & $0.395$ & $0.395$
 \end{tabular}
 \vspace{0.5em}
\caption{Low-biased estimates to the optimal stopping value of the fractional Brownian motion for different Hurst parameters $H$, obtained with linear ($\mathcal{T}_{\mathrm{sig}}$) and deep ($\mathcal{T}_{\log}$) signature stopping policies. $N$ is the truncation level of the log-/signature.
For the deep signature stopping policies we used neural networks with $I = 2$ hidden layers and $q= \eta_{2,N} +30 $ neurons.
The values in the two rightmost columns are the  low- and upper-biased estimates obtained in \cite{becker2019deep}.
The overall Monte-Carlo error in the resimulation is below $0.0004$.
The definition of the discrete time optimal stopping problem corresponding to the case $H=0.0$ was discussed in the theoretical observations above.}\label{tab:fbm_values}
\end{table}

\textbf{Evaluation of the numerical results.} We have approximated the value of the optimal stopping problem using both linear signature stopping policies ($\mathcal{T}_{\mathrm{sig}}$) and deep signature stopping policies ($\mathcal{T}_{\log}$), for Hurst parameters $H\in\{0.1, 0.2, \dots, 1.0\}$ and $H=0$ (in the sense of the theoretical observations above).
The resulting low-biased estimates are presented in Table~\ref{tab:fbm_values} (see Section~\ref{sec:lower-bounds} for a description of the calculation of low-biased estimates).
The values associated to the linear stopping policy presented in Table~\ref{tab:fbm_values} are obtained using a stochastic gradient decent method based one the form of the expected payoff given in Proposition~\ref{lem:regularized_value}.
The technical details for this procedure can be found in Section \ref{sec:non-linear-loss}.
We have used $2^{19}= 524,288$ samples for the optimization and $2^{23} = 8,388,608$ samples for the resimulation.
The values associated to the deep stopping policy were obtained using a network architecture consisting of $I=2$ hidden layers, $q = \eta_{2,N} + 30$ neurons for each layer, depending on the truncation level $N$ of the log-signature, and relu activation function.
The networks were trained with $2^{21}= 2,097,152$ samples and the loss-function presented in Section~\ref{sec:non-linear-loss}. For the resimulation we used again $2^{23}$ samples.

Comparing the values in Table 2, it is not very surprising that we observe a worse performance of the linear method compared with the non-linear method using the same truncation level.
The values obtained with the deep signature policies achieve the exact values for $H=1$ and $H=0$ (see the theoretical observations above)
and come reasonably close to he low-biased values presented \cite{becker2019deep}, already for a truncation level $N=2$.
This observation becomes more impressing when we recall from Tabel~\ref{tab:dimensions} that the dimension of the log-signature for $d=2$ and $N=2, 3, 4, 5$ is given by $\eta_{2,N}  = 3, 5, 8, 14$.
These dimensions must be seen in contrast to the $100$-dimensional input features used in \cite{becker2019deep}.
This suggests that also in the context of optimal stopping of non-Markov processes the log-signature can serve as an efficient compression of the path.

\begin{figure}
\includegraphics[width=4.2in]{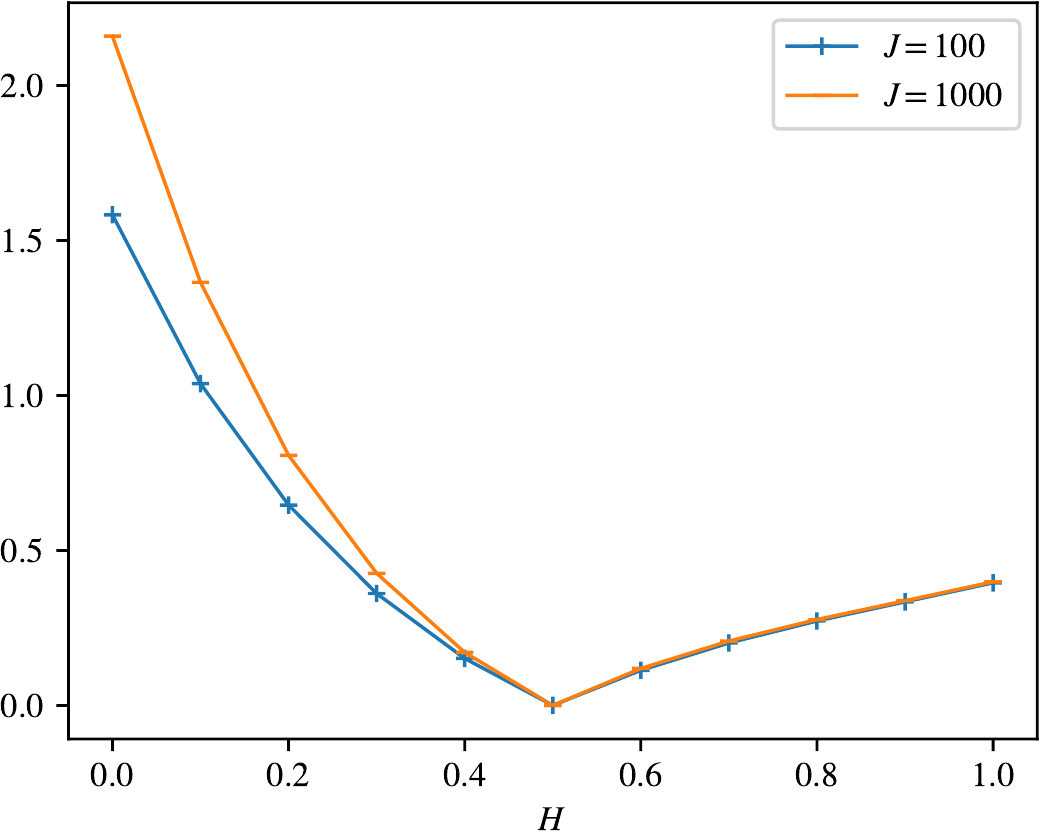}
\caption{A visualization of lower biased estimates to the optimal stopping value of fractional Brownian motion for different Hurst parameters $H\in\{0.0, 0.1, \dots, 1.0\}$, using time discretization grids with $J=100$ and $J=1000$ steps.
The values were calculated with the deep signature stopping rule using a truncation level $N=3$ for the log-signature and a neural network with $I=2$ hidden layers and $q=5 + 30$ nodes per layer.}\label{fig:fbm:figure}
\end{figure}

Figure~\ref{fig:fbm:figure} visualizes low-biased estimates to the value of optimal stopping problem of the fractional Brownian motion using time discretizations with $J=100$ and $J=1000$ steps.
The values were obtained using a deep signature stopping policy with a signature truncation level $N=3$ and the network configuration $I =2$, $q = 35$.
Note that the dimension of the (log-)signature obviously does not depend on the size of the time discretization grid.
Therefore, the number of model parameters of the signature stopping policy does not change when increasing $J$.
We see that the approximate values for $H \ll 0.5$ are significantly higher when using the finer discretization grid.
An observation of this sort was already foreshadowed by the divergence of the optimal stopping value for $H\to0$ from Proposition~\ref{prop:Hto0}.

\textbf{Linearization.} Before closing this section, we are going to discuss the numerical application of the results from Section~\ref{sec:appr-stopp-probl}.
Since the payoff coincides with the underlying path itself, Corollary~\ref{cor:opt-stopping-expected-sig} applies and after estimating the expected signature for some truncation level, this allows to approximate the value of the optimal stopping problem by a numerical approximation of the supremum in the right-hand side of \eqref{eqn:approx_stopping}.
Section~\ref{sec:linear_numerics} describes the technical details for this step.
However, it turned out that the values obtain with this procedure are drastically worse then the values obtained with any of the other methods that we have studied.
The difficulty lies in the tradeoff between the optimisation constrained $\{l \in T(V^{\ast})\;|\;\deg(l) + |l| \le K\}$ and the truncation level of the signature.
More precisely, we understand from the estimate in Lemma \ref{lem:key_shuffle_exp} that choosing $K$ large requires to also choose the truncation level $N$ large in order to decrease the error that is introduced from replacing the exponential with the shuffle exponential.
However, increasing $N$ also increases computational costs and, for this example, the limits of our computational resources were reached before obtaining reasonable results.

\subsection{An American put option in a rough electricity market}
We consider the following rough model for the spot electricity prices $S$ proposed in \cite{bennedsen2017rough}
\begin{align*}
S_t = \exp(X_t) = \exp(\Lambda_t + Z_t), \quad 0 \le t \le T,
\end{align*}
where lambda $\Lambda$ is a seasonality component and $Z$ is the rough base signal.
Note that we have refrained from using a spike signal in the model for instructive reasons.
Since we are interested in pricing an option on $S$ and in this context prices are usually corrected for seasonality influences, we will assume without much loss of generality that $$\Lambda \equiv x_0 \in \R_+.$$
Regarding the rough component, we follow the propositions in \cite{bennedsen2017rough} and finally choose the following model for the log-price
\begin{align*}
X_t = x_0 + \frac{\sigma}{c_\alpha}\int_{-\infty}^{t} (t-s)^{\alpha}e^{-\lambda (t-s)}  d W_s, \quad 0 \le t \le T,
\end{align*}
where $\alpha\in (-1/2, 1/2)$, $\lambda, \sigma > 0$, $c_\alpha$ is a normalizations constant such that $\E(X_t^{2})=\sigma^{2}$ and $W$ is a two sided standard Brownian motion.
Note that $X$ is a stationary Gaussian process and has the following auto-correlation function (see \cite{barndorff2012notes})
\begin{align*}
\mathrm{Corr}(X_t, X_{t+h}) = \frac{2^{-\alpha + \frac{1}{2}}}{\Gamma(\alpha + \frac{1}{2})}\left(\lambda h \right)^{\alpha + \frac{1}{2}}\widetilde K_{\alpha + \frac{1}{2}}(\lambda h)
\end{align*}
where $\widetilde{K}_{\nu}(x) = e^{-x}K_{\nu}(x)$ and $K_\nu$ is the modified Bessel function of the second kind.
One can easily verify that the paths of $X$ are $\gamma$-Hölder continuous for any $\gamma < \alpha + 1/2$ (see e.g. \cite{bennedsen2020semiparametric}).
Hence, by the same arguments as in Example~\ref{exmpl:fbm}, it follows that $X$ and the time augmented process $\hat X$ have a lift to a geometric rough path.

We are interested in approximating the price of an American put option on $S$, which is given by
\begin{align*}
\sup_{\tau \in \mathcal{S}}\E\left[ e^{-r\tau}\left(K - S_{\tau\wedge T}\right)_+ \right],
\end{align*}
where $K>0$ is the strike and $r\ge 0$ is an interest rate.
Hence, the payoff process of the corresponding optimal stopping problem is a continuous functional of underlying path given by $Y_t = e^{-rt}(K-\exp(X_t))_+$.

We set $T=1$ and for Monte-Carlo simulation use a time discretization with $J=100$ steps.
Note that the sample trajectories of $X$ can be obtained using the exact form of the above autocorrelation function.
In order to obtain simulations with a constant initial value, we sampled from the conditional law $\mathbb{P}(\cdot \vert X_0 = x_0)$.
For example, when generating the samples of $X$ with the Cholesky decomposition of the covariance matrix, this is simply achieved by setting the first component of the standard normal input noise to zero.

\begin{table}
\begin{tabular}{c | r | r | r | r || r }
$K$ &       $N=1$ &      $N= 2$ &     $N= 3$ &       $N=4$ & \multicolumn{1}{c}{EUR}\\
\hline
80  &    3.923 &    4.062 &    4.076 &    4.073 & 0.516\\
90  &    9.841 &   10.185 &   10.206 &   10.228 & 2.139 \\
100 &   17.753 &   18.279 &   18.341 &   18.363 & 5.770 \\
110 &   26.679 &   27.352 &   27.397 &   27.423 & 11.580 \\
120 &   36.099 &   36.854 &   36.945 &   36.898 & 19.128 \\
\end{tabular}
\vspace{0.5em}
\caption{Low-biased estimates of the value of the American put option with $T=1.0$ and different strikes $K$ on the energy spot price with model parameters $x_0 =100$, $\alpha = -0.4$, $\lambda = 0.02$, $\sigma=0.2$ and $r=0.05$.
A time discretization grid with $J=100$ points was used.
The values were obtained with deep signature stopping policies using $I=2$ hidden layer and $q=\eta_N + 30$ neurons per layer.
The rightmost column represents the prices of the European option.
The overall Monte Carlo error in the resimulation is below $0.0005$.}\label{tab:rough_electricity_values}
\end{table}

In Table~\ref{tab:rough_electricity_values} we present low-biased estimates to the value of American put option obtained with deep signature stopping policies.
As model parameters we have chosen $\alpha = -0.4$, $\lambda = 0.02$ (these values are based on statistical estimates given in \cite{bennedsen2017rough}), initial value $x_0 = 100$, volatility $\sigma = 0.2$ and interest rate $r=0.05$.
The network architecture for the deep stopping policy consisted of $I=2$ hidden layers, $q = \eta_{2, N} + 30$ neurons per layer and relu activation functions.
The networks were trained using $2^{20}$ sample paths and the low-biased estimates were calculated with $2^{23}$ sample paths.
As a reference value we have also presented an estimate of the price of the European put option, which is given by $\E[e^{-r T}(K-S_T)_+]$.
We observe that the values significantly improve when moving from truncation level $N=1$ to $N=2$.
For higher levels this improvement quickly saturates, which suggests that the values are approaching the true value of the option.

\subsection*{Acknowledgments}
\label{sec:acknowledgements}

All authors are supported by the MATH+ project AA4-2 \textit{Optimal control in energy markets using rough analysis and deep networks}.

\bibliographystyle{alpha}
\bibliography{refs}

\appendix

\section{Technical aspects of stopped rough paths}
\label{sec:techn-aspects-stopp}
Recall from Definition~\ref{defn:stopped_rp} the space stopped rough paths $\Lambda_T$ and its metric.
The following lemma characterizes the topology on $\Lambda_T$.
 \begin{lemma}\label{lem:final-topology}
  The topology on $\Lambda_T$ coincides with the final topology induced by the map $\varphi \colon [0,T] \times \hat{\Omega}_T^p \to \Lambda_T$, $\varphi(t,\hat{\mathbb{X}}) = \hat{\mathbb{X}}|_{[0,t]}$. Moreover, $\Lambda_T$ is Polish.
 \end{lemma}

 \begin{proof}
  A set $U \subset \Lambda_T$ is open with respect to the final topology if and only if $\varphi^{-1}(U)$ is open in $[0,T] \times \hat{\Omega}_T^p$. One can easily check that $\varphi$ is continuous for the topology induced by $d$, therefore $\varphi^{-1}(U)$ is open for every open set $U \subset \Lambda_T$. Now assume that $\varphi^{-1}(U)$ is open for a set $U \subset  \Lambda_T$. Let $\mathbb{X}|_{[0,t]} \in U$ and choose $\mathbb{Y}|_{[0,s]} \in \Lambda_T$ with $d(\mathbb{X}|_{[0,t]}, \mathbb{Y}|_{[0,s]}) < \varepsilon$. Our goal is to prove that $\mathbb{Y}|_{[0,s]} \in U$ for $\varepsilon$ chosen sufficiently small. Note that
  \begin{align*}
   \varphi^{-1}(\mathbb{X}|_{[0,t]}) = \{(t,\tilde{\mathbb{X}})\, :\, \tilde{\mathbb{X}}|_{[0,t]} = \mathbb{X}|_{[0,t]} \}.
  \end{align*}
  Assume $s \geq t$ first.
  Then $d(\mathbb{X}|_{[0,t]}, \mathbb{Y}|_{[0,s]}) < \varepsilon$ implies that
  \begin{align*}
   |t-s| < \varepsilon \quad \text{and} \quad d_{p-\mathrm{var};[0,s]}(\tilde{\mathbb{X}}|_{[0,s]}, \mathbb{Y}|_{[0,s]}) < \varepsilon
  \end{align*}
  where $\tilde{\mathbb{X}}|_{[0,s]}$ is the stopped path defined on $[0,s]$ as explained in Definition \ref{defn:stopped_rp}. Let $\tilde{\X} = \tilde{\mathbb{X}}|_{[0,T]} \in \hat{\Omega}_T^p$ be the stopped path defined on the whole time interval $[0,T]$.
  Since $(t,\tilde{\mathbb{X}}) \in \varphi^{-1}(\mathbb{X}|_{[0,t]}) \subset \varphi^{-1}(U)$ and $\varphi^{-1}(U)$ is open, there is a $\delta > 0$ such that whenever $u \in (t - \delta,t + \delta)$ and $d_{p-\mathrm{var};[0,T]}(\tilde{\mathbb{X}},\tilde{\mathbb{Y}}) < \delta$, we have $(u,\tilde{\mathbb{Y}}) \in \varphi^{-1}(U)$. Choosing $\varepsilon$ sufficiently small, we can assume that $s \in (t - \delta,t + \delta)$. Define $\tilde{\Y} = \tilde{\mathbb{Y}}|_{[0,T]} \in \hat{\Omega}_T^p$ as in Definition \ref{defn:stopped_rp}.
  Then $(s,\tilde{\mathbb{Y}}) \in \varphi^{-1}(\mathbb{Y}|_{[0,s]})$ and
  \begin{align*}
   d_{p-\mathrm{var};[0,T]}(\tilde{\mathbb{X}},\tilde{\mathbb{Y}}) &\leq C_p(d_{p-\mathrm{var};[0,s]}(\tilde{\mathbb{X}},\tilde{\mathbb{Y}}) + d_{p-\mathrm{var};[s,T]}(\tilde{\mathbb{X}},\tilde{\mathbb{Y}})) \\
   &= C_p d_{p-\mathrm{var};[0,s]}(\tilde{\mathbb{X}}|_{[0,s]},\mathbb{Y}|_{[0,s]}) \leq C_p \varepsilon.
  \end{align*}
  Choosing $\varepsilon$ small, we conclude $(s,\tilde{\mathbb{Y}}) \in \varphi^{-1}(U)$ and thus $\mathbb{Y}|_{[0,s]} \in U$. For $s \leq t$, we can argue similarly which proves that both topologies indeed coincide. Concerning the second statement, separability follows from the separability of $[0,T] \times \hat{\Omega}_T^p$ and the fact that $\varphi$ is a continuous surjection. To prove that $\Lambda_T$ is complete with respect to the metric $d$ is straightforward and follows from the fact that $[0,T]$ and $\hat{\Omega}_T^p$ are complete.
  \end{proof}

 \begin{corollary}\label{cor:final-continuity}
  Let $Z$ be any topological space. A map $g \colon \Lambda_T \to Z$ is continuous if and only if the map $[0,T] \times \hat{\Omega}_T^p \ni (t,\hat{\mathbb{X}}) \mapsto g(\hat{\mathbb{X}}|_{[0,t]}) \in Z$ is continuous.
 \end{corollary}

 \begin{proof}
  Follows from the universal property of the final topology.
 \end{proof}
 
\section{Solution of an optimal stopping problem corresponding to a discrete time fractional Brownian motion with zero Hurst parameter}\label{sec:H0calcuation}
Recall from Remark \ref{eq:H0covariance} that, as $H$ tends to zero, the covariance kernel of a fractional Brownian motion converges pointwise to the kernel $\rho_0$ given in \eqref{eq:H0covariance}. 
In this section we are going to derive an explicit solution to the optimal stopping problem associated to  a discrete time Gaussian process with the covariance kernel $\rho_0$.
Therefore let $(\xi_0, \xi_1, \dots, \xi_J)$ be a finite sequence of independent standard normal distributed random variables defined on a probability space $(\Omega, \mathcal{F}, \P)$.
Define the process 
\begin{align*}
Y = (Y_j)_{j=0,1, \dots, J}, \quad Y_j = \xi_j - \xi_0,
\end{align*}
and note that its covariance structure is given by $\E[Y_i Y_j] = 2\cdot\rho_0(i, j)$.
For ease of notation we refrain from scaling the process by $1/\sqrt{2}$.
Further note that the filtration that the process $Y$ generates is given by
\begin{align*}
\mathcal{F}_0 = \{\emptyset, \Omega\}, \quad \mathcal{F}_j = \sigma(\xi_1 - \xi_0, \dots, \xi_j - \xi_0).
\end{align*}
We are interested in finding the solution to the discrete time optimal stopping problem associated to the payoff $Y$.
Therefore denote by $\mathcal{S}_J$ the set of discrete stopping times with respect to the filtration $(\mathcal{F}_j)_{j = 0, 1, \dots, J}$ and define the value process
\begin{align*}
V_j = \sup_{\sigma \in \mathcal{S}_J, \sigma \ge j} \E\left[Y_\sigma \Big\vert \mathcal{F}_j \right], \quad j = 0, 1, \dots, J.
\end{align*}
Then we have $V_J = Y_J$ and by the dynamic programming principle it follows
\begin{align}\label{eqn:bellman}
V_j = \max\left\{Y_j,\; \E[ V_{j+1} \vert \mathcal{F}_j]\right\}, \quad j = 0, 1, ..., J-1.
\end{align}
We are now going to proceed by recursively calculating the conditional expectation in the dynamic programming principle. Therefore define
\begin{align*}
\overline{\xi_{1:j}} := \frac{1}{j}\sum_{k=1}^{j}\xi_k
\end{align*}
and note that we have the following
\begin{lemma}
Let $j \in \{1, \dots, J\}$, then the conditional expectation of $\xi_0$ given $\mathcal{F}_j$ is given by
\begin{align*}
\E[\xi_0 \vert \mathcal{F}_j] = \frac{j}{j+1}(\xi_0 - \overline{\xi_{1:j}}).
\end{align*}
\end{lemma}
\begin{proof}
First note that the random variable in the above right-hand side is mean-zero Gaussian and is $\mathcal{F}_j$-measurable which can easily be seen from the identity
\begin{align*}
\frac{j}{j+1}(\xi_0 - \overline{\xi_{1:j}}) = \frac{1}{j+1}\sum_{i=1}^{j} Y_i.
\end{align*}
On the other hand, we have for all $A \in \mathcal{F}_j$
\begin{align*}
\E\left[ \left( \xi_0 -\frac{j}{j+1}(\xi_0 - \overline{\xi_{1:j}})\right)1_{A}\right] = \E\left[ \frac{1}{j+1}\left( \xi_0 + \xi_1 + \cdots + \xi_j\right)1_{A}\right] = 0
\end{align*}
since the random variable $\xi_0 + \xi_1 + \cdots + \xi_j$ is independent of $\mathcal{F}_j$.
Indeed, every $Y_i = \xi_i - \xi_0$ for $i \le j$ is uncorrelated with $\xi_0 + \xi_1 + \cdots + \xi_j$, hence due to Gaussianity independent.
\end{proof}
\begin{remark}
We can observe that the process $(Y_j, \xi_0 - \overline{\xi_{1:j}})_{j=0, 1, \dots, J}$ is a two-dimensional Markov process adapted to the filtration $(\mathcal{F}_j)_{j = 0, 1, \dots, J}$.
Indeed, recall that $Y_j = \xi_j - \xi_0$ and note that $\xi_j$ is independent of $\mathcal{F}_i$ for all $i<j$.
Further, we see that the conditional law of $\xi_0$ given $\mathcal{F}_i$ is Gaussian.
From the above lemma it follows that the conditional mean is given by $\xi_0 - \overline{\xi_{1:i}}$ and following a simple Gaussian calculation we see that the conditional variance is deterministic (depending only on $i$ and $j$).
\end{remark}
The conditional expectation of the value function $V_J$ with respect to $\mathcal{F}_{J-1}$ is then given by
\begin{align*}
\E[V_J \vert \mathcal{F}_{J-1}] = \E[- \xi_0 \vert \mathcal{F}_{J-1}] = \frac{J-1}{J}(\overline{\xi_{1:J-1}} - \xi_0).
\end{align*}
Hence, plugging into the Bellman equation \eqref{eqn:bellman} we obtain for the value process at $J-1$
\begin{align*}
V_{J-1} &=\max\left\{\xi_{J-1} - \xi_0,   \frac{J-1}{J}(\overline{\xi_{1:J-1}} - \xi_0) \right\}.
\end{align*}
Using the dynamic programming principle we are now going to show the following
\begin{proposition}\label{prop:H0} For all $j \in \{1, \dots, J-1\}$ it holds
\begin{align*}
V_{j} &=\max\left\{\xi_{j} - \xi_0,   \mu_{j+1} + \frac{j}{j+1}(\overline{\xi_{1:j}} - \xi_0) \right\},
\end{align*}
and $V_0 = \mu_1$,
where
\begin{align*}
\mu_{J} = 0, \quad \mu_{j} = \sqrt{\frac{j}{j+1}}\gamma\left(\mu_{j+1}\sqrt{\frac{j+1}{j}}\right), \quad j = J-1, J-2, \dots, 1
\end{align*}
where $\gamma(x) = \phi(x) + x\Phi(x)$ and $\Phi$ and $\phi$ are the cdf and pdf of the standard normal distribution.
\end{proposition}
\begin{remark}\label{rem:H0Jinfty} Note that for any $j\in\{1,\dots,J-1\}$ it holds $\mu_j \ge \gamma(2\cdot \mu_{j+1})/2$.
Hence it follows that $\mu_1 \ge f^{\circ (J-1)}(0)$ with $f(x):= \gamma(2x)/2$.
Since $f$ is continuous and $f(x) > x$ for all $x\in[0, \infty)$ it follows that $f^{\circ J}(0) \to \infty$ for $J\to \infty$.
This confirms what one intuitively expects, that the value of the above optimal stopping problem converges to infinity as $J$ tends to infinity.
\end{remark}
\begin{proof}
The claim was already shown for $j = J-1$.
We are going to use a backwards induction to proof the claim for $0 \le j < J-1$.
In order to explicitly calculate the conditional expectation of $V_{j+1}$ given $\mathcal{F}_{j}$, we are going the derive a orthogonal decomposition of each of the random variables in the above maximum with respect to $\mathcal{F}_{j}$.
For the first random variable we have
\begin{align*}
\xi_{j+1} - \xi_0 = \left(\xi_{j+1} - \frac{1}{j+1}\xi_0-\frac{j}{j+1}\overline{\xi_{1:j}} \right) + \frac{j}{j+1}(\overline{\xi_{1:j}} -\xi_0)
\end{align*}
and for the second we have
\begin{align*}
\frac{j+1}{j+2}(\overline{\xi_{1:j+1}} - \xi_0) =& \frac{1}{j+2}\xi_{j+1} + \frac{j}{j+2}\overline{\xi_{1:j}} - \frac{j+1}{j+2}\xi_0\\
=& \frac{1}{j+2}\xi_{j+1} + \frac{j}{j+2}(\overline{\xi_{1:j}} - \xi_0) -\frac{1}{j+2}\xi_0\\
=& \frac{1}{j+2}\xi_{j+1} + \frac{j}{j+2}(\overline{\xi_{1:j}} - \xi_0) + \frac{1}{j+2}\frac{j}{j+1}(\overline{\xi_{1:j}} -\xi_0)\\
&- \frac{1}{j+2}\left(\frac{1}{j+1}\xi_0+\frac{j}{j+1}\overline{\xi_{1:j}}\right)\\
=& \frac{1}{j+2}\left(\xi_{j+1} - \frac{1}{j+1}\xi_0-\frac{j}{j+1}\overline{\xi_{1:j}} \right) + \frac{j}{j+1}(\overline{\xi_{1:j}} -\xi_0).
\end{align*}
For simplifying notation in what follows, we define
\begin{align*}
\eta_j := \sqrt{\frac{j+1}{j+2}} \left(\xi_{j+1} - \frac{1}{j+1}\xi_0-\frac{j}{j+1}\overline{\xi_{1:j}} \right)
\end{align*}
which is a standard Gaussian random variable orthogonal to $\mathcal{F}_j$, indeed note that
\begin{align*}
\mathbb{V}\left(\xi_{j+1} - \frac{1}{j+1}\xi_0-\frac{j}{j+1}\overline{\xi_{1:j}}\right) = 1+ \frac{1}{(j+1)^{2}}+\frac{j}{(j+1)^{2}} = \frac{j+2}{j+1}.
\end{align*}
Using this orthogonal decomposition and the induction claim we then have
\begin{align*}
\E[V_{j+1}\vert \mathcal{F}_{j}] = \sqrt{\frac{j+2}{j+1}}\E\left[\max\left\{\eta_j,\;\, \mu_{j+2}\sqrt{\frac{j+1}{j+2}} +  \frac{1}{j+2}\eta_j \right\}\right] 
+ \frac{j}{j+1}(\overline{\xi_{1:j}} -\xi_0).
\end{align*}
Resolving the above expectation is now simple one-variate Gaussian calculation. Therefore let $0 \le \alpha < 1$, $\beta \in \R$, then note that
\begin{align*}
\E[\max\{\eta_j, \beta + \alpha \eta_j\}] &= \int_{\beta/(1-\alpha)}^{\infty} x \phi(x)dx +\int_{-\infty}^{\beta/(1-\alpha)} (\beta + \alpha x) \phi(x)dx \\
&= \phi(\beta/(1-\alpha)) + \beta\Phi(\beta/(1-\alpha)) - \alpha\phi(\beta/(1-\alpha)) \\
&= (1-\alpha)\phi(\beta/(1-\alpha)) + \beta\Phi(\beta/(1-\alpha)) \\
&= (1-\alpha)\gamma(\beta/(1-\alpha)).
\end{align*}
Therefore we finally have
\begin{align*}
\E[V_{j+1}\vert \mathcal{F}_{j}] = \sqrt{\frac{j+2}{j+1}} \frac{j+1}{j+2}\gamma\left(\mu_{j+2}\sqrt{\frac{j+1}{j+2}}\frac{j+2}{j+1} \right)
+ \frac{j}{j+1}(\overline{\xi_{1:j}} -\xi_0)
\end{align*}
and the rest of the claim follows from the dynamical programming principle.
\end{proof}

\section{Implementation details}\label{sec:implmentation_details}
We are going to discuss the technical steps which are necessary in order to allow Monte-Carlo simulation for the signature stopping methodology.
Throughout this section we will assume that the rough path $\X\in\Omega_T$ is the lift of a $d$-dimensional process $X=(X_t)_{0 \le t \le T}$ and that $Y=(Y_t)_{0 \le t \le T}$ is a continuous real-valued process adapted to the filtration generated by $X$.

\subsection{Time discretization}\label{sec:time_discretization}
We fix a time grid with equidistant points $(t_j)_{j=0,1,\dots, J}$, $t_j = Tj/J$ and define the discretized payoff process $\tilde Y$ by
$$\tilde Y_j =  Y_{t_j}, \quad j=0, 1, \dots, J.$$
Next, we fix a refinement $(s_j)_{j=0, 1, \dots, J^{\prime}}$ of the grid $(t_j)_{j=0,1,\dots, J}$ with $J^{\prime} \ge J$ and denote by $\mathcal{S}_J$ the set of discrete stopping times with respect to the following filtration
\begin{align*}
\tilde{\mathcal{F}}_j := \sigma\left(X_{s_i}\,:\, s_i\le t_j,\, i = 0, 1, \dots J^{\prime}\right), \quad j = 0, 1, \dots, J. 
\end{align*}
We can approximate the value of time optimal stopping problem \eqref{eq:intro-optimal-stopping} from below by the value of  discrete time optimal stopping problem associated to $(\tilde Y, (\tilde{\mathcal{F}}_j)_{j=0, 1, \dots, J})$, i.e.
\begin{align*}
\sup_{\sigma \in \mathcal{S}_{J}}\E\Big[ \tilde Y_{\sigma \wedge J}\Big] \le \sup_{\tau \in \mathcal{S}}\E\Big[ Y_{\tau \wedge T}\Big].
\end{align*}
Indeed, since $\tilde{\mathcal{F}}_{j} \subset \mathcal{F}_{t_j}$ for all $j \in\{0, 1, \dots, J\}$, it is easy to see that for any $\sigma\in\mathcal{S}_J$, $\omega \mapsto t_{\sigma(\omega)}$ defines a stopping time in $\mathcal{S}$.
Next, denote by $\tilde X$ the linear interpolation of the process $X$ on the grid $(s_j)_{j=0, 1, \dots, J^{\prime}}$, i.e.
\begin{align*}
\tilde X_t = X_{s_j} + \frac{t- s_j}{s_{j+1}-s_j}(X_{s_{j+1}} - X_{s_{j}}), \quad s_j \le t \le s_{j+1}, \; j = 0, 1, \dots, J^{\prime}.
\end{align*}
Clearly also the time augmented path $((t, \tilde X_t))_{0 \le t \le T}$ is piecewise linear and this process has a (trivial) lift to a rough path, which we denote by $\tilde{\X}$.
We will now adapt the definition of the randomized stopping times \eqref{def:rand-stopping-times} to the discrete time setting.
Therefore, let $Z$ be positive random variable as in Definition~\ref{def:rand-stopping-times} and let  $\theta\in\mathcal{T}$ be a continuous stopping rule, then we define 
\begin{align}\label{eq:discrete_stopping_time}
\sigma_\theta^{r} = \inf\left\{ 0 \le j\le J \; \Bigg\vert\; \sum_{i = 0}^{j} \theta(\tilde\X\vert_{[0, t_i]})^{2} \ge Z \right\}.
\end{align} 
Since $\tilde\X\vert_{[0, t_j]}$ is $\tilde{\mathcal{F}}_j$-measurable for all $j\in\{0, 1, \dots, J\}$, it follows that that $\sigma_\theta^{r} \in \mathcal{S}_J$.
With minor changes in proof, an analogous result of Proposition~\ref{lem:regularized_value} holds for the stopping times $\sigma_\theta^{r}$ and yields the following expression for the expected payoff
\begin{align}\label{eq:discretized_regular_value}
\E\left[ \tilde{Y}_{\sigma^{r}_\theta}\right] = \E\left[ \tilde{Y}_0 + \sum_{j = 0}^{J-1}G(j)(\tilde Y_{j+1} - \tilde Y_{j})\right], \qquad\text{with } G(j) := 1 - F_Z\left(\sum_{i=0}^{j}\theta(\tilde\X\vert_{[0, t_i]})^{2} \right),
\end{align}
where $F_Z$ is the cumulative distribution function of $Z$.

\subsection{Simulation of the (log-)signature}\label{sec:signature-sampling}
The standard, model independent procedure for approximating the truncated signature of given simulations of a path is to calculate the truncated signature of the linearly interpolated path.
Regarding the optimal stopping problem, it is therefore necessary to generate joint samples of the process $X$ and the payoff $Y$, while one may possibly choose a finer approximation grid for the simulation of $X$. 
This may require discretization of stochastic / rough differential equations.
The truncated signature of the linearly interpolated path on the given time grid can then be calculated exactly (up to floating point errors) and therefore always yields an element in the free nilpotent Lie-group.
Fortunately, packages for this task are readily available. See, for instance, the \emph{iisignature} library \cite{RG20}.
Note that in case the signature is needed on each subinterval $[0, s_j]$ (see Section~\ref{sec:non-linear-loss} below), one can use the calculation of the signature on the preceding subinterval to significantly improve the efficiency. 
As already mentioned in Section \ref{sec:non-linear-rules}, the log-signature of the linearly interpolated path can be directly calculated in the Hall basis and does not require the calculation of the signature.

\subsection{Optimization of linear signature stopping policies based on the expected signature}\label{sec:linear_numerics}
If the setting allows to pose the optimization problem in terms of the expected signature only, as in the context of Corollary~\ref{cor:opt-stopping-expected-sig}, then we need to compute the expected truncated signature $\E[ \hat\X_{0,T}^{\le N}]$.
This can be done using Monte-Carlo estimation and the procedure for calculating samples of the truncated signature as in explaind in Section~\ref{sec:signature-sampling} above.
The big advantage is, of course, that we only need to compute the expected signature once, and can then apply the optimization algorithm of our choice to a deterministic optimization problem.
For calculating the expression $\sexp(- (l \shuffle l)\blue{1})$ upto truncation degree $N$, one needs to implement a structure that allows the handling of noncommutative polynomials with variable coefficients and that supports the shuffle operation.
Since we could not find an appropriate package we have implemented this part ourself. 
Now let $\lambda_0 w_0 + \lambda_1 w_1 + \cdots + \lambda_{n_k} w_{n_k} \in T^{N}((\R^{d+1})^\ast)$, where $w_0, \dots, w_{n_k}$ ranges over all words in of length $k \le N$ in the alphabet $\mathcal{A}_{d+1}$ (a total number of $n_k + 1= ((d+1)^{k+1} -1)/d)$ words).
Further, let $E \in T^{N}(\R^{d+1})$ be a Monte-Carlo estimate of the expected truncated signature  signature, then the expression that needs to maximized is given by
$$\left\langle \sexp\left( - \left(\lambda_0 w_0 + \cdots + \lambda_{n_k} w_{n_k}\right)^{\shuffle 2} \blue{1}\right) l_Y, E \right\rangle,$$
where $l_Y \in T((\R^{d+1})^{*})$, such that $Y = \langle l_Y, \hat\X^{<\infty}\rangle$.
This expression is a polynomial in the coefficients $\lambda_1, \dots, \lambda_{n_k}$.
We can then apply a general state-of-the-art iterative optimization algorithms to optimize the above polynomial subject to the constraint $|\lambda_1| + \cdots + |\lambda_{n_k}| \le K$.
The resulting optimal coefficients can further be used to calculate a low-biased estimate of the optimal stopping problem as described in Section~\ref{sec:lower-bounds} below.
  
\subsection{A loss-function for linear and deep signature stopping policies}\label{sec:non-linear-loss} 
In this subsection we will define a loss-function for the linear and deep stopping policies based on the regular form of the expected payoff \eqref{eq:discretized_regular_value}.
The loss-function can be used with a stochastic gradient descent method in order to numerically optimize the parameters of the stopping policy and thus to approximate the value of the optimal stopping problem.
Assume that we have generated $M$ sample trajectories $$(\tilde X_j^{(m)}, \tilde{Y}_j^{(m)})_{0 \le j \le J,\; 1 \le m \le M}, \qquad\text{with }
 \tilde X_j^{(m)}\in \R^{d+1}, \quad\tilde Y_j^{(m)}\in \R, $$ of the time augmented path $\hat{X}$ and the payoff process $Y$ (for simplicity we have not used a finer time grid for $\hat X$).
Fix a truncation level $N \in \{1, 2, \dots\}$.
For each series of samples $\tilde X^{(m)}_0, \dots, \tilde X^{(m)}_J$ we can calculated the truncated signature of the linearly interpolated path on each subinterval $[0, t_j]$, yielding the the sequence of vectors
$$(\tilde{\X}_j^{(m)})_{0 \le j \le J, \; 1 \le m \le M}, \qquad \text{with }\tilde{\X}_j^{(m)} \in T^{N}(\R^{d+1}) \cong \R^{1+\sigma_{d+1, N}}.$$
Let $\theta_l \in \mathcal{T}_{\mathrm{sig}}$ be a signature stopping policy of the form $\theta_l(\hat\X) = \langle l, \hat\X^{< \infty} \rangle$
for some $l \in T^{N}((\R^{d+1})^{\ast})$.
Further let $\mathcal{M} \subset \{1, ..., M\}$ be the indices of a batch of samples. Using \eqref{eq:discretized_regular_value} we define the following loss-function for the stopping policy $\theta_l$
\begin{align*}
\ell\left( \tilde{Y}\vert_{\mathcal{M}}, \tilde{\L}\vert_{\mathcal{M}}; \; l\right) = - \frac{1}{|\mathcal{M}|}\sum_{m \in \mathcal{M}}\left\{\tilde{Y}_0^{(m)} + \sum_{j=0}^{J-1} G_Z\left(\sum_{i=0}^{j}\left\langle l, \tilde{\X}^{(m)}_{i}\right\rangle^{2}\right)(\tilde{Y}_{j+1}^{(m)}- \tilde{Y}_{j}^{(m)}) \right\},
\end{align*}
where $G_Z = 1 - F_Z$.

Next, let $\theta = \theta_{\log} \circ \tlog \in\mathcal{T}_{\log}$ be a deep stopping policy with a neural network $\theta_{\log}$ of the form \eqref{eq:networks}.
In this case, the definition of the loss-function is in principle the same as above, however, the log-signature can be calculated directly in the hall basis, yielding the sequence of vectors
$$(\tilde{\L}_j^{(m)})_{0 \le j \le J, \; 1 \le m \le M}, \qquad \text{with }\tilde{\L}_j^{(m)} \in \mathfrak{g}^{N}(\R^{d+1}) \cong \R^{\eta_{d+1, N}}.$$
Again using \eqref{eq:discretized_regular_value} we define the following loss-function for the deep stopping policy $\theta$ by
\begin{align*}
\ell\left( \tilde{Y}\vert_{\mathcal{M}}, \tilde{\L}\vert_{\mathcal{M}}; \theta_{\log} \right) = - \frac{1}{|\mathcal{M}|}\sum_{m \in \mathcal{M}}\left\{\tilde{Y}_0^{(m)} + \sum_{j=0}^{J-1} G_Z\left(\sum_{i=0}^{j}\theta_{\log}\left(\tilde{\L}^{(m)}_{i}\right)^{2}\right)(\tilde{Y}_{j+1}^{(m)}- \tilde{Y}_{j}^{(m)}) \right\}.
\end{align*}
Regarding the distribution of $Z$, we have tried the exponential distribution and the log-logistic distribution.
The latter, corresponding to $G_Z(x) = 1/(1+x)$, led to initially quicker learning rates, however for the overall performance it was better to use the exponential distribution.

\subsection{Calculation of lower-bounds}\label{sec:lower-bounds}
Assume that we have obtained a stopping policy $\theta^{\ast} \in \mathcal{T}$ as the output of a numerical optimization.
For example $\theta^{\ast} = \langle l^{\ast}, \cdot \rangle \in\mathcal{T}_{\mathrm{sig}}$ for some numerically optimized $l^{\ast} \in T(V^{\ast})$; or $\theta^{\ast}\in \mathcal{T}_{\log}$ with some trained neural network $\theta_{\log}^{*}$.
We then calculate a estimate to the  value of the optimal stopping problem by a Monte-Carlo approximation of \eqref{eq:discretized_regular_value} with a new set of sample trajectories.
Note that this estimates is precisely given by the respective loss-function of Section \ref{sec:non-linear-loss} evaluated at the new set of samples.
The estimate is low biased due to the sub-optimality of the stopping policy.

Alternatively, note that the stopping policy $\theta^{\ast}$ defines a stopping time  $\sigma^{\ast} \coloneqq \sigma^{r}_{\theta^{\ast}}\in\mathcal{S}_J$ by \eqref{eq:discrete_stopping_time}.
Therefore, after generating independent samples of $Z$, one for each sample trajectory of $X$ and $Y$, we obtain another low-biased estimator of the optimal stopping value by the Monte-Carlo approximation of $\E[\tilde{Y}_{\sigma^{*}}]$.
Since both estimates yield the same result, upto a the additional Monte-Carlo error introduced by the sampling of $Z$, the latter value can serve for a sanity check.
\end{document}